\documentclass[a4paper,10pt]{amsart}
\usepackage{pb-diagram}
\usepackage{xypic}
\usepackage{amscd}
\usepackage [curve]{xy}
\usepackage{bm}
\usepackage{amsmath,amssymb,amscd,bbm,amsthm,mathrsfs,hyperref}
\usepackage{latexsym}
\usepackage{amsfonts}
\usepackage{amsthm}
\usepackage{indentfirst}
 \newtheorem{thm}{Theorem}[section]
 \newtheorem{cor}[thm]{Corollary}
 \newtheorem{lem}[thm]{Lemma}
 \newtheorem{prop}[thm]{Proposition}
 \newtheorem{defn}[thm]{Definition}
 \newtheorem{ex}[thm]{Example}
 \newtheorem{rem}[thm]{Remark}

 \newcommand{\Hom}{\mathrm{Hom}}

\title{Homologically smooth connected cochain DGAs}
\author{X.-F. Mao}
\address{Department of Mathematics, Shanghai University, Shanghai, China, 200444}
\address{Newtouch center for Mathematics of Shanghai University, Shanghai, China, 200444}
\email{xuefengmao@shu.edu.cn}

\date{}

\begin{document}
 \def\abstactname{abstract}
\begin{abstract}
Let $\mathscr{A}$ be a connected cochain DG algebra such that $H(\mathscr{A})$ is a Noetherian graded algebra. We give some criteria for $\mathscr{A}$ to be homologically smooth in terms of the singularity category, the cone length of the canonical module $k$ and the global dimension of $\mathscr{A}$. For any cohomologically finite DG $\mathscr{A}$-module $M$, we show that it is compact when $\mathscr{A}$ is homologically smooth. If $\mathscr{A}$ is in addition Gorenstein, we get
 $$\mathrm{CMreg}M = \mathrm{depth}_{\mathscr{A}}\mathscr{A} + \mathrm{Ext.reg}\, M<\infty,$$
where
 $\mathrm{CMreg}M$ is the Castelnuovo-Mumford regularity of $M$, $\mathrm{depth}_{\mathscr{A}}\mathscr{A}$ is the depth of $\mathscr{A}$ and $ \mathrm{Ext.reg}\, M$ is the Ext-regularity of $M$.
\end{abstract}



\maketitle

\section*{introduction}
Over the past two decades, the introduction and application of DG
homological methods and techniques have been one of the main areas
in homological algebra. In DG homological algebra, the homologically smoothness of a DG algebra plays a similar important role as
the regularity of an algebra does in the homological ring theory.
The research on this fundamental property of DG algebras have attracted many people's interests.
 In \cite{HW1},  He-Wu introduced the concept of
Koszul DG algebras, and obtained a DG version of the Koszul duality for Koszul, homologically smooth and Gorenstein DG algebras.
The author and Wu \cite{MW2} proved that any homologically
smooth connected cochain DG algebra $\mathscr{A}$ is cohomologically unbounded
unless $\mathscr{A}$ is quasi-isomorphic to the simple algebra $k$. And it was
proved that the $\mathrm{Ext}$-algebra
 of a homologically smooth DG algebra $\mathscr{A}$
is Frobenius if and only if both $\mathscr{D}^b_{lf}(\mathscr{A})$ and
$\mathscr{D}^b_{lf}(\mathscr{A}\!^{op})$ admit Auslander-Reiten triangles.  In \cite{Sh}, Shklyarov developed
a Riemann-Roch Theorem for homologically smooth DG algebras.
Besides these, some important classes of DG algebras are homologically smooth. For example, Calabi-Yau DG algebras introduced by Ginzburg in \cite{Gin} are homologically smooth by definition. Especially, non-trivial Noetherian DG down-up algebras and DG free algebras generated by two degree $1$ elements are  Calabi-Yau DG algebras by \cite{MHLX} and \cite{MXYA}, respectively. Moreover, there is a construction called `Calabi-Yau completion' \cite{Kel2} which produces a canonical Calabi-Yau DG algebra from a homologically smooth DG algebra.

One sees from above that it is  meaningful to study homologically smooth DG algebras thoroughly.
A feasible way to study an algebra is via various homological
invariants of the modules on them.  There have been many
kinds of invariants on DG module since Appasov's PhD thesis \cite{Apa},
where he defined homological dimensions of DG
modules from both resolutional and functorial
points of view. Frankild and J$\o$rgensen \cite{FJ} introduced and studied
$k$-projective dimension and $k$-injective dimension for DG modules
over a local chain DG algebra. Later, Yekutieli-Zhang \cite{YZ}
introduced projective dimension $\mathrm{proj.dim}_{\mathscr{A}}M$ and flat
dimension $\mathrm{flat.dim}_{\mathscr{A}}M$ for a DG module $M$ over a
homologically bounded DG algebra $\mathscr{A}$. Any one of these
invariants for DG modules can be seen as a generalization of the
corresponding classical homological dimensions of modules over a
ring. However, it seems that none of them can be used to define a
finite global dimension of a DG algebra. Inspired from the
definition of free class for differential modules over a commutative
ring in \cite{ABI},  the invariant DG free class for semi-free DG
modules was introduced in \cite{MW3}. Recall that the DG free class a semi-free DG $\mathscr{A}$-module
is defined to be the shortest length of all its strictly increasing semi-free filtrations. For any DG $\mathscr{A}$-module,
the least DG free classes of all its semi-free
resolutions is called cone length.  This invariant of DG modules plays a similar role as
projective dimension of modules does in homological ring theory.
It is well known in homological ring theory that the projective dimension of a module over a local ring is equal to the length of its minimal projective resolution.  In this paper, we prove the following theorem (see Theorem \ref{mindggf}).
\\
\begin{bfseries}
Theorem \ A.
\end{bfseries}
Let $M$ be an object in $\mathscr{D}^{+}(\mathscr{A})$ such that $\mathrm{cl}_{\mathscr{A}}M<\infty$. Then there is a minimal semi-free resolution $G$ of $M$ such that $\mathrm{DGfree.class}_{\mathscr{A}}G=\mathrm{cl}_{\mathscr{A}}M$.

In
\cite{Jor1}, J$\o$rgensen put forward a question on how to define
global dimension of DG algebras. As explained in \cite{MW3}, it is
reasonable to some degree to define left (resp. right) global
dimension of a connected DG algebra $\mathscr{A}$ to be the supremum of the
set of the cone lengthes of all DG $\mathscr{A}$-modules (resp.
$\mathscr{A}^{op}$-modules).
In
classical theory of homological algebra, it is well known that the
regular property of a commutative noetherian local ring can be
characterized by the finiteness of its global dimension and projective dimensions for all
finitely generated modules. By \cite{DGI}, we know a commutative noetherian local ring is regular if and only if every homologically finite complex is small in the derived category.  It is natural to ask whether we can get analogous results in DG setting. The following theorem  (see Theorem \ref{mainres}) confirm this positively.
\\
\begin{bfseries}
Theorem \ B.
\end{bfseries}
Let $\mathscr{A}$ be a connected cochain DG algebra such that $H(\mathscr{A})$ is a Noetherian graded algebra. Then the following statements are equivalent:

$(a)\,\, \mathscr{A}$ is homologically smooth.

$(b)\,\, \mathrm{cl}_{\mathscr{A}^e}\mathscr{A}<\infty$.

$(c)\,\, l.\mathrm{Gl.dim}\,\mathscr{A}<\infty$.

$(d)\,\, \mathscr{D}^c(\mathscr{A})=\mathscr{D}_{fg}(\mathscr{A}) $.

$(e)\,\, \mathscr{D}_{sg}(\mathscr{A})=0$.

$(f)\,\, \mathrm{cl}_{\mathscr{A}}k<\infty$.

$(g)\,\, k\in \mathscr{D}^c(\mathscr{A})$.

 Here, $\mathscr{D}_{fg}(\mathscr{A})$ and $\mathscr{D}^c(\mathscr{A})$ are the
full triangulated subcategories of the derived category  of DG $\mathscr{A}$-modules consisting
  of cohomologically finite DG $\mathscr{A}$-modules and compact DG $\mathscr{A}$-modules, respectively.
  Note that compact DG $\mathscr{A}$-modules are just small objects in $\mathscr{D}(\mathscr{A})$, and $\mathscr{D}_{sg}(\mathscr{A})$ is the singularity category $\mathscr{D}_{fg}(\mathscr{A})/\mathscr{D}^c(\mathscr{A})$.

 In \cite{Jor2}, J$\o$rgensen introduced Dwyer-Greenlees theory to differential graded homological algebra and developed a duality between $\mathscr{D}_{fg}(\mathscr{A})$ and $\mathscr{D}_{fg}(\mathscr{A}^{op})$ under the hypothesis \cite[Setup $4.1$]{Jor2} and the
  the additional condition that $H(\mathscr{A})$ is Noetherian with a balanced dualizing complex. Applying Theorem B, one sees that
  $\mathscr{D}_{fg}(\mathscr{A})=\mathscr{D}^c(\mathscr{A})$ when $\mathscr{A}$ is homologically smooth and $H(\mathscr{A})$ is Noetherian.
  This leads straightforwardly to the following duality:
\begin{align*}
\xymatrix{&\mathscr{D}_{fg}(\mathscr{A})\quad\quad\ar@<1ex>[r]^{R\Hom_{\mathscr{A}}(-,
\mathscr{A})}&\quad\quad
\mathscr{D}_{fg}(\mathscr{A}\!^{op})\ar@<1ex>[l]^{R\Hom_{\mathscr{A}\!^{op}}(-,\mathscr{A})}}.
\end{align*}

The Ext-regularities and Castelnuovo-Mumford regularities for DG modules were introduced
by J$\o$rgensen in \cite{Jor2}. Under the assumptions mentioned above, he obtained some interesting results on these two invariants for DG modules  in $\mathscr{D}_{fg}(\mathscr{A})$ (see \cite[Theorem 5.7]{Jor2}). In this paper, we show the following theorem (See Theorem \ref{formula}).
\\
\begin{bfseries}
Theorem \ C.
\end{bfseries}
Let $\mathscr{A}$ be a Gorenstein and homologically smooth connected cochain DG algebra such that $H(\mathscr{A})$ is a Noetherian graded algebra. Then for any object $M$ in $\mathscr{D}_{fg}(\mathscr{A})$, we have
$$\mathrm{CMreg}M = \mathrm{depth}_{\mathscr{A}}\mathscr{A}+ \mathrm{Ext.reg}\, M<\infty.$$

\section{preliminaries}
In this section, we review some basics on differential graded (DG
for short) homological algebra, whose main main novelty is the study
of the internal structure of a category of DG modules from a point
of view inspired by classical homological algebra. There is some
overlap here with the papers \cite{MW1, MW2,FHT2}. It is assumed that
the reader is familiar with basics on the theory of triangulated
categories and derived categories. If this is not the case, we refer
to  \cite{Nee, Wei} for more details on them.

Throughout the paper, $k$ is a fixed field.  Let $\mathscr{A}$ be a $\Bbb{Z}$-graded
$k$-algebra.  If there is a $k$-linear map $\partial_{\mathscr{A}}: \mathscr{A}\to \mathscr{A}$  of
degree $1$ such that $\partial_{\mathscr{A}}^2 = 0$ and
\begin{align*}
\partial_{\mathscr{A}}(ab) = \partial_{\mathscr{A}}(a)b + (-1)^{n|a|}a\partial_{\mathscr{A}}(b)
\end{align*}
for all graded elements $a, b\in \mathscr{A}$, then  $\mathscr{A}$ is called a cochain
differential graded $k$-algebra. We write DG for differential graded.
For any cochain DG $k$-algebra $\mathscr{A}$, its
underlying graded algebra obtained by forgetting the differential of
$\mathscr{A}$ is denoted by $\mathscr{A}^{\#}$. If $\mathscr{A}^{\#}$ is a connected graded
algebra, then $\mathscr{A}$ is called a connected cochain DG algebra.

For the rest of this paper, we denote $\mathscr{A}$ as a
connected DG algebra over a field $k$ if no special assumption is
emphasized. The cohomology graded algebra of $\mathscr{A}$ is the graded algebra $$H(\mathscr{A})=\bigoplus_{i\in \Bbb{Z}}\frac{\mathrm{ker}(\partial_{\mathscr{A}}^i)}{\mathrm{im}(\partial_{\mathscr{A}}^{i-1})}.$$
 For any cocycle element $z\in \mathrm{ker}(\partial_{\mathscr{A}}^i)$, we write $\lceil z \rceil$ as the cohomology class in $H(\mathscr{A})$ represented by $z$. It is easy to
check that $H(\mathscr{A})$ is a connected graded algebra if $\mathscr{A}$ is a
connected DG algebra. We denote $\mathscr{A}\!^{op}$ as the opposite DG
algebra of $\mathscr{A}$, whose multiplication is defined as
 $a \cdot b = (-1)^{|a|\cdot|b|}ba$ for all
homogeneous elements $a$ and $b$ in $\mathscr{A}$. For any connected cochain DG algebra $\mathscr{A}$, it has the following maximal DG
ideal $$\frak{m}: \cdots\to 0\to \mathscr{A}^1\stackrel{\partial_1}{\to}
\mathscr{A}^2\stackrel{\partial_2}{\to} \cdots \stackrel{\partial_{n-1}}{\to}
\mathscr{A}^n\stackrel{\partial_n}{\to}
 \cdots .$$ Obviously, the enveloping DG algebra $\mathscr{A}^e = \mathscr{A}\otimes \mathscr{A}\!^{op}$ of $\mathscr{A}$
is also a connected DG algebra with $H(\mathscr{A}^e)\cong H(\mathscr{A})^e$, and its
maximal DG ideal is $\frak{m}\otimes \mathscr{A}^{op} + \mathscr{A}\otimes
\frak{m}^{op}$.

 A left DG module over $\mathscr{A}$  (DG $\mathscr{A}$-module for
short) is a graded
$\mathscr{A}^{\#}$-module together with a linear $k$-map $\partial_M: M\to M$
of degree $1$ satisfying the Leibniz rule:
\begin{align*}
\partial_{M}(am)=\partial_{\mathscr{A}}(a)m + (-1)^{|a|}a\partial_{M}(m),
\end{align*}
for all graded elements $a \in \mathscr{A}, \, m \in M$. For any  left DG $\mathscr{A}$-module, it is well known that $H(M)$ is a left graded $H(\mathscr{A})$-module. We say that a DG $\mathscr{A}$-module is acyclic if $H(M)=0$.
A right DG module
over $\mathscr{A}$ is defined similarly. It is easy to check
that any right DG modules over $\mathscr{A}$ can be identified with DG
$\mathscr{A}^{op}$-modules. For any DG $\mathscr{A}$-module $M$ and $i\in \Bbb{Z}$, the $i$-th
suspension of $M$ is the DG $\mathscr{A}$-module $\Sigma^i M$ defined by
$(\Sigma^iM)^j = M^{j+i}$. If $m \in M^l,$ the corresponding element
in $(\Sigma^i M)^{l-i}$ is denoted by $\Sigma^i m$. We have
$a\Sigma^i m = (-1)^{|a|i}\Sigma^i (am) $ and $\partial_{\Sigma^i
M}(\Sigma^i m) = (-1)^{i}\Sigma^i \partial_M(m)$, for any graded
elements $a\in \mathscr{A}, m\in M$.

An $\mathscr{A}$-homomorphism $f:M \to N$
of degree $i$ between  DG $\mathscr{A}$-modules $M$ and $N$ is a $k$-linear map of degree $i$ such that
\begin{align*}
f(am) = (-1)^{i\cdot|a|}af(m),\,\,\, \textrm{for all}\,\,\, a\in
\mathscr{A}, m\in M.
\end{align*}
Denote $\Hom_{\mathscr{A}}(M,N)$ as the graded vector space of all graded
$\mathscr{A}$-homomorphisms from $M$ to $N$. This is a complex with them
differential $\partial_{\Hom}$ defined by
\begin{align*}
\partial_{\Hom}(f)=\partial_{N}\circ f -
(-1)^{|f|}f\circ\partial_{M}
\end{align*}
for all $f\in \Hom_{\mathscr{A}}(M, N)$.
 A morphism of DG $\mathscr{A}$-modules
from $M$ to $N$ is an $\mathscr{A}$-homomorphism $f$ of degree $0$ such that
$\partial_{N}\circ f = f\circ \partial_{M}$. The induced map $H(f)$
of $f$ on the cohomologies is then a morphism of left graded $H(\mathscr{A})$-modules. If
$H(f)$ is an isomorphism, then $f$ is called a
quasi-isomorphism, which is denoted as $f: M
\stackrel{\simeq}{\to}N.$ Let $f$ and $g$ be two morphisms of DG $\mathscr{A}$-modules between $M$ and $N$. If there is an $\mathscr{A}$-homomorphism $\sigma: M\to N$ of degree $-1$ such that $f-g=\partial_N\circ \sigma + \sigma \circ \partial_M$, then we say that $f$ and $g$ are homotopic to each other and we write $f\sim g$.
  A DG $\mathscr{A}$-module $M$ is called homotopically trivial if $\mathrm{id}_M\sim 0$.
  A morphism $f: M\to N$ of DG $\mathscr{A}$-modules is called a homotopy equivalence if there is a morphism $h:N\to M$ such that $f\circ h \sim \mathrm{id}_N$ and $h\circ f\sim \mathrm{id}_M$. And $h$ is called a homotopy inverse of $f$. One sees easily that any homotopy equivalence is a quasi-isomorphism.

  A DG $\mathscr{A}$-module $P$ (resp. $I$) is called $K$-projective (resp. $K$-injective) if  the functor $\Hom_{\mathscr{A}}(P, -)$ (resp. $\Hom_{\mathscr{A}}(-,I)$)
preserves quasi-isomorphisms.  And a DG $\mathscr{A}$-module $F$ is called  $K$-flat if the functor $-\otimes_{\mathscr{A}} F$ preserves quasi-isomorphisms. A $K$-projective resolution (resp. $K$-flat resolution)
 of a DG $\mathscr{A}$-module $M$ is a
quasi-isomorphism $\theta: P \to M$, where $P$ is a $K$-projective (resp. $K$-flat) DG $\mathscr{A}$-module. Similarly, a $K$-injective resolution of $M$ is defined as a quasi-isomorphism $\eta: M\stackrel{\simeq}{\to} I$, where $I$ is a $K$-injective DG $\mathscr{A}$-module.
A DG $\mathscr{A}$-module is called DG free, if it is isomorphic to a
direct sum of suspensions of $\mathscr{A}$ (note it is not a free object in
the category of DG modules). Let $Y$ be a graded set, we denote
$\mathscr{A}^{(Y)}$ as the free DG module $\oplus_{y\in Y}\mathscr{A} e_y$, where
$|e_y|=|y|$ and $\partial(e_y)=0$. Let $M$ be a DG $\mathscr{A}$-module. A
subset $E$ of $M$ is called a semi-basis if it is a free basis of
$M^{\#}$ over $\mathscr{A}^{\#}$ and has a decomposition $E =
\bigsqcup_{i\ge0}E_i$ as a union of disjoint graded subsets $E_i$
such that
\begin{align*}
\partial(E_0)=0 \,\,\, \textrm {and} \,\,\,\partial(E_u)\subseteq A
(\bigsqcup_{i<u}E_i)\, \,\,\textrm{for all}\,\, \,u >0.
\end{align*} A DG $\mathscr{A}$-module $F$ is called semi-free if there is a
sequence of DG submodules
\begin{align*}
0=F_{-1}\subset F_{0}\subseteq\cdots\subseteq F_{n}\subset\cdots
\end{align*}
such that $F = \cup_{n \ge 0} \,F_{n}$ and that each
$F_{n}/F_{n-1}=\mathscr{A}\otimes V(n)$ is a DG free $\mathscr{A}$-module. The
differential of $F$ can be decomposed as $\partial_F = d_0 + d_1 +
\cdots$, where $d_0 =
\partial_{\mathscr{A}}\otimes \mathrm{id}$ and each $d_i, i\ge 1$ is an $\mathscr{A}$-linear map
satisfying $d_i(V(l))\subseteq \mathscr{A}^{\#}\otimes V(l-i)$. It is easy to
check that a DG $\mathscr{A}$-module is semi-free if and only if it admits a
semi-basis.
A semi-free resolution of a DG $\mathscr{A}$-module $M$ is a
quasi-isomorphism $\varepsilon:F \to M,$ where $F$ is a semi-free
DG $\mathscr{A}$-module. Sometimes, we just say that $F$ is a semi-free
resolution of $M$. Semi-free resolutions play a similar important
role in DG homological algebra as ordinary free resolutions do in
homological ring theory.

Let $\mathscr{C}(\mathscr{A})$ be the category of DG $\mathscr{A}$-modules and morphisms of DG $\mathscr{A}$-modules.  The
derived category of $\mathscr{C}(\mathscr{A})$ is
denoted by $\mathscr{D}(\mathscr{A})$, which is constructed from
$\mathscr{C}(\mathscr{A})$ by inverting quasi-isomorphisms.  The right derived
functor of $\Hom$, is denoted by $R\Hom$, and the
 left derived functor of $\otimes$, is denoted by $\,{}^L\otimes$.
They can be computed via  K-projective, K-injective and K-flat resolutions of DG modules.
It is easy to check that $\Hom_{\mathscr{D}(\mathscr{A})}(M,N) = H^0(R\Hom_{\mathscr{A}}(M,N))$, for any objects $M, N$ in $\mathscr{D}(\mathscr{A})$.  A DG $\mathscr{A}$-module
is called compact, if the functor
$\Hom_{\mathscr{D}(\mathscr{A})}(M,-)$ preserves all coproducts in
$\mathscr{D}(\mathscr{A})$. By \cite[Theorem 5.3]{Kel1}, a DG $\mathscr{A}$-module $M$ is
compact, if and only if it is in the smallest triangulated thick
subcategory of $\mathscr{D}(\mathscr{A})$ containing ${}_{\mathscr{A}}\mathscr{A}$. To use the
language of topologists, a DG $\mathscr{A}$-module is compact if it can be
built finitely from ${}_{\mathscr{A}}\mathscr{A}$, using suspensions and distinguished
triangles.

 For any  DG $\mathscr{A}$-module $M$, it is called cohomologically finite if  $H(M)$ is a finitely generated $H(\mathscr{A})$-module.
 We say that $M$ is cohomologically locally finite if each $\dim_kH^i(M)<\infty$.
Let $\mathscr{D}_{fg}(\mathscr{A})$ and $\mathscr{D}_{lf}(\mathscr{A})$ be the
full triangulated subcategories of $\mathscr{D}(\mathscr{A})$ consisting
  of cohomologically finite DG $\mathscr{A}$-modules and cohomologically
locally finite DG $\mathscr{A}$-modules, respectively.
If the graded $H(\mathscr{A})$-module $H(M)$ is bounded below (resp. bounded above), we say that $M$ is cohomologically bounded below (resp. cohomologically bounded above).
Let $\mathscr{D}^{+}(\mathscr{A})$ (resp. $\mathscr{D}^{-}(\mathscr{A})$) be the
full triangulated subcategory of $\mathscr{D}(\mathscr{A})$ consisting of cohomologically bounded above (resp. cohomologically bounded below)
DG $\mathscr{A}$-modules. One sees easily that $\mathscr{D}^{-}(\mathscr{A})\cap \mathscr{D}^{+}(\mathscr{A})$ consists of DG $\mathscr{A}$-modules with bounded cohomology. It is natural to write
 $\mathscr{D}^b(\mathscr{A})=\mathscr{D}^{-}(\mathscr{A})\cap \mathscr{D}^{+}(\mathscr{A})$.
Obviously, we have inclusions $\mathscr{D}^c(\mathscr{A})\subseteq \mathscr{D}_{fg}(\mathscr{A})\subseteq \mathscr{D}^+(\mathscr{A}).$  Following \cite{Buc,Orl,Chen, Kel3},
the singularity category of $\mathscr{A}$ is defined as the Verdier quotient $\mathscr{D}_{sg}(\mathscr{A})=\mathscr{D}_{fg}(\mathscr{A})/\mathscr{D}^c(\mathscr{A})$.
It is confirmed by the results of \cite{Orl} that that singularity category of an algebra measures the degree to which the algebra is `singular'. One of the motivations of this paper is to seek a similar result for DG version.
\section{some basic lemmas}
In this section, we will give some fundamental lemmas on semi-free resolutions,  isomorphisms and compact DG modules.
Semi-free resolutions play an important role in DG homological
algebra as ordinary free resolutions do in homological ring theory. The following lemma indicates that any DG $\mathscr{A}$-module has a semi-free resolution.
\begin{lem}\label{homotopycom}\cite[Proposition 6.6]{FHT2}
For any DG algebra $\mathscr{A}$, each DG $\mathscr{A}$-module $M$ admits a semi-free resolution $f:
F\stackrel{\simeq}{\to} M$. If $g: G\stackrel{\simeq}{\to} M$ is
a second semi-free resolution, then there is a homotopy equivalence
$h: G\to F$ such that $g\sim f\circ h$.
\end{lem}
\begin{lem}\label{homotopyproj}\cite[Proposition 6.4]{FHT2}
For any DG algebra $\mathscr{A}$, if $F$ is a semi-free DG $\mathscr{A}$-module and $\eta: M\to N$ is
a quasi-isomorphism,  then $\Hom_{\mathscr{A}}(F,\eta)$ is a quasi-isomorphism.
Equivalently, the functor $\Hom_{\mathscr{A}}(F,-)$ maps quasi-trivial DG $\mathscr{A}$-modules to acyclic complexes.
Hence any semi-free DG $\mathscr{A}$-module is K-projective.
\end{lem}

Let $\mathscr{A}$ be a connected cochain DG algebra. A semi-free DG $\mathscr{A}$-module $F$ is {\sl minimal} if
$\partial_{F}(F)\subseteq \mathrm{m}F$. The minimality of $F$
implies that both $\Hom_{\mathscr{A}}(F,k)$ and $k\otimes_{\mathscr{A}}F$ have vanishing
differentials. As to the existence of the minimal semi-free
resolution of a DG $\mathscr{A}$-module, we have the following lemma.

\begin{lem}\label{exist}\cite[Proposition 2.4]{MW1}
Let $\mathscr{A}$ be a connected cochain DG algebra.
If $M$ is a DG $\mathscr{A}$-module in $\mathscr{D}^+(\mathscr{A})$ with
$b=\inf\{j|H^j(M)\neq 0\}$,  then there exists a minimal semi-free
resolution $F_M$ of $M$ with $F_M^{\#}\cong \coprod\limits_{i\ge
b}\Sigma^{-i}(\mathscr{A}^{\#})^{(\Lambda^i)}$, where each $\Lambda^i$ is an
index set.
\end{lem}

\begin{lem}\label{comp}\cite[Proposition 3.3]{MW2}
Let $\mathscr{A}$ be a connected cochain DG algebra.
If $M$ is a DG $\mathscr{A}$-module in $\mathscr{D}^+(\mathscr{A})$,  then $M$ is compact if and only if $\dim_kH(k\otimes_{\mathscr{A}}M)=\dim_kH(R\Hom_{\mathscr{A}}(M,k))<\infty$
\end{lem}

\begin{rem}\label{charcomp}
Let $M$ be an object in $\mathscr{D}^+(\mathscr{A})$. By Lemma \ref{exist} and Lemma \ref{comp}, one sees easily that
$M$ is compact if and only if
it admits a minimal
semi-free resolution $F_M$ which has
a finite semi-basis. To use the language of topologists, a DG $\mathscr{A}$-module is compact if it
can be  built finitely from ${}_{\mathscr{A}}\mathscr{A}$, using suspensions and distinguished
triangles. Compact DG modules play the same role as finitely
presented modules of finite projective dimension do in ring theory. One sees that compact DG modules are just small objects in $\mathscr{D}(\mathscr{A})$.
\end{rem}

\begin{lem}\cite[Remark 20.1]{FHT2}\label{projfree}
Any bounded below projective graded module over a connected graded
algebra is a free graded module.
\end{lem}

 Since any DG $\mathscr{A}$-module is a graded
$\mathscr{A}^{\#}$-module by forgetting its differential, we can easily get
the following lemma by the graded version of Nakayama Lemma.
\begin{lem}\label{naklem}\emph{(DG Nakayama Lemma)}
Let $\mathscr{A}$ be a connected cochain DG algebra.
If $M$ is a bounded below DG $\mathscr{A}$-module and $L$ is a DG
$\mathscr{A}$-submodule of $M$ such that $L+\frak{m}M=M$, then $L=M$.
\end{lem}

\begin{lem}\label{homiso}
Let $\mathscr{A}$ be a connected cochain DG algebra.
Suppose that $F$ is a bounded below DG $\mathscr{A}$-module such that
$\partial_F(F)\subseteq mF$ and $F^{\#}$ is a projective
$\mathscr{A}^{\#}$-module. If a DG morphism $\alpha: F\to F$ is homotopic to
the identity morphism $\mathrm{id}_F$,  then $\alpha$ is an
isomorphism.
\end{lem}
\begin{proof}
Since $\alpha\simeq \mathrm{id}_F$, there is a homotopy map $h: F\to
F$ such that $$\alpha -\mathrm{id}_F=h\circ \partial_F +
\partial_F\circ h.$$ Let $\overline{F}=k\otimes_{\mathscr{A}}F,
\overline{\alpha}=k\otimes_{\mathscr{A}} \alpha$ and $\overline{h}=k\otimes_{\mathscr{A}}h$.
Since $\partial_F(F)=\frak{m}F$, we have
$\overline{\alpha}=\mathrm{id}_{\overline{F}} + \overline{h}\circ
\partial_{\overline{F}}+\partial_{\overline{F}}\circ
\overline{h}=\mathrm{id}_{\overline{F}}$. Hence
$F=\mathrm{im}(\alpha)+\frak{m}F$. By Lemma \ref{naklem}, we have
$\mathrm{im}(\alpha) = F$. Since $F^{\#}$ is a projective
$\mathscr{A}^{\#}$-module, the short exact sequence
$$0\to \mathrm{ker}(\alpha)\to F\stackrel{\alpha}{\to}F\to 0$$ is linearly split.
Note that a short exact sequence of DG $\mathscr{A}$-modules is called linearly split if it is split as a short exact sequence of graded $\mathscr{A}^{\#}$-modules.
Acting $k\otimes_{\mathscr{A}}-$ on this linearly split short exact sequence, gives a short exact sequence $$0\to k\otimes_{\mathscr{A}}\mathrm{ker}(\alpha)\to \overline{F}\stackrel{\overline{\alpha}}{\to}\overline{F}\to 0$$ of graded $k$-verctor spaces. Since $\overline{\alpha}$ is a monomorphism, we have $$\mathrm{ker}(\alpha)/\frak{m}\mathrm{ker}(\alpha) = k\otimes_{\mathscr{A}} \mathrm{ker}(\alpha) =0.$$ Suppose that $\mathrm{ker}(\alpha)\neq 0$, then $\mathrm{ker}(\alpha)$ is a bounded below DG $\mathscr{A}$-module since it is a DG $\mathscr{A}$-submodule of $F$. This implies that $\mathrm{ker}(\alpha)\neq \frak{m}\mathrm{ker}(\alpha)$. It contradicts with $\mathrm{ker}(\alpha)/\frak{m}\mathrm{ker}(\alpha) = 0$. Hence $\mathrm{ker}(\alpha)=0$.
\end{proof}

\begin{lem}\label{homotopyinverse}

For any DG algebra $\mathscr{A}$ and morphism of DG $\mathscr{A}$-modules $f:M\to N$,  if there are DG morphisms $g: N\to M$ and $g':N\to M$ such that
$g'\circ f\sim \mathrm{id}_M$ and $f\circ g\sim \mathrm{id}_N$, then $f$ is a homotopy equivalence and $g$ is a homotopy inverse of $f$.
\end{lem}
\begin{proof}
By assumptions, we have $g'\sim g'\circ \mathrm{id}_N\sim g'\circ (f\circ g)=(g'\circ f)\circ g\sim \mathrm{id}_M\circ g=g$.
So $g\circ f \sim g'\circ f\sim \mathrm{id}_M$. Hence $f$ is a homotopy equivalence and $g$ is its homotopy inverse.

\end{proof}

\begin{lem}\label{homotopytrivial}
For any DG algebra $\mathscr{A}$ and any DG $\mathscr{A}$-module $M$,
the DG module $M$ is homotopically trivial if
and only if $H(\Hom_{\mathscr{A}}(M,M))=0$.
\end{lem}
\begin{proof}
If $M$ is homotopically trivial, then $\Hom_{\mathscr{A}}(M,M)$ is homotopically trivial since the functor $\Hom_{\mathscr{A}}(M,-)$ is additive. So $H(\Hom_{\mathscr{A}}(M,M))=0$.

Conversely, suppose that $H(\Hom_{\mathscr{A}}(M,M))=0$, we need to prove that
$M$ is homotopically trivial. Since
$\partial_{\Hom}(\mathrm{id}_M)=\partial_M\circ \mathrm{id}_M -
\mathrm{id}_M\circ \partial_M =0$, there is $\sigma\in \Hom_{\mathscr{A}}(M,M)$
of degree $-1$ such that
$\mathrm{id}_M=\partial_{\Hom}(\sigma)=\partial_M\circ \sigma +
\sigma\circ
\partial_M$. Therefore, $M$ is homotopically trivial.
\end{proof}

\begin{lem}\label{dgfree}Let $\mathscr{A}$ be a connected cochain DG algebra.
Assume that $\mathscr{A}$ is a DG free $\mathscr{A}$-module with a direct summand $P$ such that
$H(P)$ is bounded below.  Then $P$ is also a
 DG free $\mathscr{A}$-module.
\end{lem}
\begin{proof}
Let $F = \bigoplus\limits_{i\in I}\mathscr{A}e^i$. By the assumption, $H(P)$
is a direct summand of $H(F)$, which is a free graded
$H(\mathscr{A})$-module. Hence $H(P)$ is a projective $H(\mathscr{A})$-module. Since
$H(P)$ is bounded below and $H(\mathscr{A})$ is
connected, $H(P)$ is a free $H(\mathscr{A})$-module. Let $H(P) =
\bigoplus\limits_{j\in J}H(\mathscr{A})[f^j]$, where each $f^j$ is a cocycle
in $P$.

Let $L = \bigoplus\limits_{j\in J}\mathscr{A}x_j$ be the DG free
$\mathscr{A}$-module with a cocyle basis $\{x_j|j\in J\}$. We define a
morphism of DG $\mathscr{A}$-modules: $\varepsilon: L \to P$ by $\varepsilon
(x^j) = f^j$, for any $j\in J$. It is easy to check that
$\varepsilon$ is a quasi-isomorphism. Since both $L$ and $P$ are
$K$-projective, $\varepsilon$ is a homotopy equivalence. Hence
$\varepsilon$ is an isomorphism as both $L$ and $P$ are minimal.
\end{proof}

\begin{lem}\label{semidecomp}Let $\mathscr{A}$ be a connected cochain DG algebra.
If $F$ is a semi-free DG $\mathscr{A}$-module such that $H(F)$ is bounded
below, then there is a minimal semi-free resolution $G$ of $F$ such that $F \cong G\oplus
Q$ as a DG $\mathscr{A}$-module, where $Q$ is
homotopically trivial DG $\mathscr{A}$-submodule of $F$.
\end{lem}
\begin{proof}
By Lemma \ref{exist}, $F$ admits a minimal semi-free resolution
$g:G\to F$ with $\inf\{i|G^i\neq 0\} =\inf\{i|H^i(F)\neq 0\}$. Since
$F$ can be considered as a semi-free resolution of itself, there is a
homotopy equivalence $h:G\to F$ such that $\mathrm{id}_F\circ h\sim
g$ by Lemma \ref{homotopycom}. Let $f: F\to G $ be the homotopy
inverse of $h$. Then $f\circ h \sim \mathrm{id}_G$. Hence there is
an $\mathscr{A}$-linear homomorphism $\sigma: G\to G$ of degree $-1$ such that
$f\circ h -\mathrm{id}_G =
\partial_G\circ \sigma + \sigma \circ \partial_G $. Since
$\partial_G(G)\subseteq \frak{m}G$ and $\sigma$ is $\mathscr{A}$-linear, we
have $f\circ h-\mathrm{id}_G\subseteq \frak{m}G$. Hence
$\overline{f\circ h}=k\otimes_{\mathscr{A}}(f\circ h)$ is the identity map of
$$\overline{G} =G/\frak{m}G= k\otimes_{\mathscr{A}}G.$$ Acting on the exact
sequence $$G\stackrel{f\circ h}{\longrightarrow} G\longrightarrow
\mathrm{coker} (f\circ h)\longrightarrow 0$$ by $k\otimes_{\mathscr{A}}-$ gives
a new exact sequence
$$\overline{G}\stackrel{\overline{f\circ h}}{\longrightarrow}\overline{G}\longrightarrow
\overline{\mathrm{coker}(f\circ h)}\longrightarrow 0.$$ This implies
that $\overline{\mathrm{coker}(f\circ h)}=0$. Hence
$\mathrm{coker}(f\circ h) =\frak{m}\cdot\mathrm{coker}(f\circ h)$.
If $\mathrm{coker}(f\circ h) = G/\mathrm{im}(f\circ h)$ is not zero,
then it is bounded below since $G$ is bounded below. Let $v =
\inf\{i|(\mathrm{coker}(f\circ h))^i\neq 0\}$. Since $\frak{m}$ is
concentrated in degrees $\ge 1$, $\frak{m}\cdot
\mathrm{coker}(f\circ h)$ is concentrated in degrees $\ge v+1$. This
contradicts with $\mathrm{coker}(f\circ h)
=\frak{m}\cdot\mathrm{coker}(f\circ h)$. Therefore,
$\mathrm{coker}(f\circ h) =0$ and $f\circ h$ is surjective. We have
the following linearly split short exact sequence
\begin{equation*}\label{shortexact}
0\longrightarrow \mathrm{ker}(f\circ h)\longrightarrow
G\stackrel{f\circ h}{\longrightarrow} G\longrightarrow 0 \eqno{(1)}.
\end{equation*}
Note that a short exact sequence of DG $\mathscr{A}$-modules is called linearly split if it is split as a short exact sequence of graded $\mathscr{A}^{\#}$-modules.
Acting on $(1)$ by $k\otimes_{\mathscr{A}}-$ gives a new short exact sequence
$$0\longrightarrow \overline{\mathrm{ker}(f\circ h)}\longrightarrow
\overline{G}\stackrel{\overline{f\circ h}}{\longrightarrow}
\overline{G}\longrightarrow 0.$$ This implies that
$\overline{\mathrm{ker}(f\circ h)}=0$ since $\overline{f\circ h}$ is
the identity map. Hence $\mathrm{ker}(f\circ
h)=\frak{m}\cdot\mathrm{ker}(f\circ h)$. If $\mathrm{ker}(f\circ h)$
is not zero, then it is bounded below since it is a DG $\mathscr{A}$-submodule
of $G$. Let $u=\inf\{i|(\mathrm{ker}(f\circ h))^i\neq 0\}$. Then
$\frak{m}\cdot \mathrm{ker}(f\circ h)$ is concentrated in degrees
$\ge u+1$. This contradicts with $\mathrm{ker}(f\circ
h)=\frak{m}\cdot\mathrm{ker}(f\circ h)$. Thus $\mathrm{ker}(f\circ
h) =0$ and $f\circ h$ is an isomorphism. Let $\theta:G\to G$ be the
inverse of $f\circ h$. Then $\theta\circ f\circ h = \mathrm{id}_G$.
This implies that $h$ is a monomorphism and the short exact sequence
$$0\longrightarrow G\stackrel{h}{\longrightarrow} F\longrightarrow
\mathrm{coker}(h)\longrightarrow 0$$ is split. Hence $F\cong G\oplus
\mathrm{coker}(h)$ as a DG $\mathscr{A}$-module. One sees that $\mathrm{coker}(h)$ is quasi-trivial
since $H(h)$ is an
isomorphism. By Lemma
\ref{homotopyproj}, both $\Hom_{\mathscr{A}}(F,\mathrm{coker}(h))$ and
$\Hom_{\mathscr{A}}(G,\mathrm{coker}(h))$ are acyclic. Since
$$\Hom_{\mathscr{A}}(F,\mathrm{coker}(h))\cong \Hom_{\mathscr{A}}(G,\mathrm{coker}(h))\oplus
\Hom_{\mathscr{A}}(\mathrm{coker}(h),\mathrm{coker}(h)),$$
we have
$H(\Hom_{\mathscr{A}}(\mathrm{coker}(h),\mathrm{coker}(h)))=0$. By Lemma
\ref{homotopytrivial}, the DG $\mathscr{A}$-module $\mathrm{coker}(h)$ is homotopically trivial.

\end{proof}

\begin{lem}\label{imporcong}
Let $\mathscr{A}$ be a connected cochain DG algebra.
Suppose that $M$, $N$ are two DG $\mathscr{A}$-modules and $X$ is a DG $\mathscr{A}^e$-module.
Then the chain map
\begin{align*}
\phi: \Hom_{\mathscr{A}}(X\otimes_{\mathscr{A}}M,& N) \to \Hom_{\mathscr{A}^e}(X, \Hom_k(M,N))\\
&f\mapsto \phi(f):x\mapsto f_x:m\to f(x\otimes m)
\end{align*}
is an isomorphism.
\end{lem}
\begin{proof}
It suffices to prove the following two statements:

(1) For any $f\in \Hom_{\mathscr{A}}(X\otimes_{\mathscr{A}}M, N)$,
we need to show $\phi(f)$ is $\mathscr{A}^e$-linear.

(2) The map $\phi$ is a chain map and has an inverse chain map.

For any $a\otimes b\in \mathscr{A}^e$, $x\in X$ and $m\in M$, we have
\begin{align*}
\phi(f)[(a\otimes b)x](m)&=f_{(a\otimes b)x}(m)=f[(a\otimes b)x\otimes m]\\
                        &=(-1)^{|b|\cdot |x|}f[axb\otimes m]=(-1)^{|b|\cdot |x|+|f|\cdot |a|}af[xb\otimes m]
\end{align*}
and
\begin{align*}
[(a\otimes b)\phi(f)(x)](m)&=[(a\otimes b)f_{x}](m)=(-1)^{|f|\cdot |b|+|x|\cdot |b|}af_x(bm) \\
&=(-1)^{|f|\cdot |b|+|x|\cdot |b|}af(x\otimes bm).
\end{align*}
Hence $\phi(f)[(a\otimes b)x]=(-1)^{|f|\cdot (|b|+|a|)}(a\otimes b)\phi(f)(x)$. We prove statement (1).

For any $f\in \Hom_{\mathscr{A}}(X\otimes_{\mathscr{A}}M, N)$, $x\in X$ and $m\in M$, we have
\begin{align*}
&\quad [\partial_{\Hom}\circ \phi(f)](x)(m)= [\partial_{\Hom_{k}(M,N)}\circ \phi(f)-(-1)^{|f|}\phi(f)\circ \partial_X](x)(m)\\
&=[\partial_N\circ f_x-(-1)^{|f|+|x|}f_x\circ \partial_M](m)-(-1)^{|f|}f_{\partial_X(x)}(m)\\
&=\partial_N[f(x\otimes m)]-(-1)^{|f|+|x|}f[x\otimes \partial_M(m)]-(-1)^{|f|}f[\partial_X(x)\otimes m]\\
\end{align*}
and
\begin{align*}
&\quad [\phi\circ \partial_{\Hom}(f)](x)(m)=[\phi(\partial_N\circ f-(-1)^{|f|}f\circ \partial_{\otimes})](x)(m) \\
&=(\partial_N\circ f-(-1)^{|f|}f\circ \partial_{\otimes})_x(m)=(\partial_N\circ f-(-1)^{|f|}f\circ \partial_{\otimes})(x\otimes m)\\
&=\partial_N[f(x\otimes m)]-(-1)^{|f|}f[\partial_X(x)\otimes m]-(-1)^{|f|+|x|}f[x\otimes \partial_M(m)].
\end{align*}
Thus $\phi$ is a chain map. It remains to show that $\phi$ has an inverse chain map. We
define
\begin{align*}
\psi: \Hom_{\mathscr{A}^e}(X, \Hom_k(M,N))& \to  \Hom_{\mathscr{A}}(X\otimes_{\mathscr{A}}M, N)\\
&g\mapsto \psi(g):x\otimes m \mapsto  g(x)(m).\\
\end{align*}
We need to show the following statements:

(3)For any $g\in\Hom_{\mathscr{A}^e}(X, \Hom_k(M,N))$,  $\psi(g)$ is $\mathscr{A}$-linear.

(4)$\psi$ is a chain map and $\psi$ is the inverse of $\phi$.

For any $a\in \mathscr{A}$ and $x\otimes m\in X\otimes_{\mathscr{A}}M$, we have
\begin{align*}
\psi(g)[a(x\otimes m)]&=\psi(g)(ax\otimes m)=g(ax)(m)\\
&=a(g(x))(m)=a[g(x)(m)]=a[\psi(g)(x\otimes m)].
\end{align*}
Hence $\psi(g)$ is $\mathscr{A}$-linear and we get $(3)$. For any $g\in \Hom_{\mathscr{A}^e}(X, \Hom_k(M,N))$ and $x\otimes m\in X\otimes_{\mathscr{A}}M$, we have
\begin{align*}
[\partial_{\Hom}\circ \psi(g)]&(x\otimes m)=[\partial_N\circ \psi(g)-(-1)^{|g|}\psi(g)\circ \partial_{\otimes}](x\otimes m)\\
                                        &=\partial_N[g(x)(m)]-(-1)^{|g|}\psi(g)[\partial_X(x)\otimes m+(-1)^{|x|}x\otimes \partial_M(m)]\\
                                        &=\partial_N[g(x)(m)]-(-1)^{|g|}g(\partial_X(x))(m)-(-1)^{|g|+|x|}g(x)[\partial_M(m)]
\end{align*}
and
\begin{align*}
&\quad [\psi\circ \partial_{\Hom}(g)](x\otimes m)= [\psi(\partial_{\Hom_k(M,N)}\circ g-(-1)^{|g|}g\circ \partial_X)](x\otimes m)\\
&=(\partial_{\Hom_k(M,N)}\circ g-(-1)^{|g|}g\circ \partial_X)(x)(m)\\
&=[\partial_N\circ g(x)-(-1)^{|g|+|x|}g(x)\circ \partial_M](m)-(-1)^{|g|}g(\partial_X(x))(m)\\
&=\partial_N[g(x)(m)]-(-1)^{|g|+|x|}g(x)[\partial_M(m)]-(-1)^{|g|}g(\partial_X(x))(m).
\end{align*}
So $\phi$ is a chain map. Furthermore,
\begin{align*}
[\phi\circ\psi(g)(x)](m)=\psi(g)_x(m)=\psi(g)(x\otimes m)=g(x)(m)
\end{align*}
and
\begin{align*}
[\psi\circ \phi(f)](x\otimes m)=[\phi(f)(x)](m)=f_x(m)=f(x\otimes m).
\end{align*}
Hence $\psi$ is the inverse of $\phi$.
\end{proof}

\begin{lem}\label{imporquasi}
Let $\mathscr{A}$ be a connected cochain DG algebra. If
$M$ is a DG $\mathscr{A}$-module such that $\dim_kH(M)<\infty$, then
for any DG $\mathscr{A}$-module $N$, the DG $\mathscr{A}^e$-module $\Hom_k(M, N)$ is quasi-isomorphic to $N\otimes
M^*$.
\end{lem}

\begin{proof}
Since $\dim_kH(M)<\infty$, $M$ is a compact DG $k$-module. Then we have
$$\Hom_k(M,N)=\Hom_k(M,k\otimes N)\cong \Hom_k(M,k)\otimes N =  M^*\otimes N\cong N\otimes M^*$$ in $\mathscr{D}(\mathcal{A}^e)$. More precisely,
the morphism \begin{align*}
\theta:\quad & N\otimes
M^* \to  \Hom_k(M,N) \\
&n\otimes f \mapsto (m\mapsto nf(m))
\end{align*}
is a quasi-isomorphism.
\end{proof}

\begin{lem}\label{aelem}
Let $\mathscr{A}$ be a connected cochain DG algebra. If
 $M$ is a DG $\mathscr{A}$-module such that $\dim_kH(M)<\infty$,  then
for any DG $\mathscr{A}$-module $N$, we have $$H(R\Hom_{\mathscr{A}^e}(\mathscr{A}, N\otimes_k
M^*))\cong H(R\Hom_{\mathscr{A}}(M,N)).$$
\end{lem}
\begin{proof}
Let $F_M$  be a semi-free resolution of $M$ and let $I_N$ be a
$K$-injective resolution of $N$. The DG $\mathscr{A}^e$-module $\mathscr{A}$ has a minimal semi-free resolution $X
\stackrel{\simeq}{\to} \mathscr{A}$. As a DG $\mathscr{A}$-module, $X$ is $K$-projective
since $$\Hom_{\mathscr{A}}(X,-)\cong \Hom_{\mathscr{A}}(\mathscr{A}^e\otimes_{\mathscr{A}^e}X,-) \cong
\Hom_{\mathscr{A}^e}(X,\Hom_{\mathscr{A}}(\mathscr{A}^e,-)) $$ and $\mathscr{A}^e=\mathscr{A}\otimes \mathscr{A}^{op}$ is a $K$-projective DG
$\mathscr{A}$-module.  We have
\begin{align*}
H(R\Hom_{\mathscr{A}}(M,N)) & \cong H(\Hom_{\mathscr{A}}(F_M,N))\\
&\cong H(\Hom_{\mathscr{A}}(F_M,I_N)) \\
&\cong H(\Hom_A(X\otimes_AF_M,I_N))\\
&\stackrel{(a)}{\cong}H(\Hom_{{\mathscr{A}}^e}(X, \Hom_k(F_M,I_N)))\\
&\stackrel{(b)}{\cong} H(\Hom_{{\mathscr{A}}^e}(X,I_N\otimes_kF_M^*))\\
&\cong H(\Hom_{{\mathscr{A}}^e}(X,N\otimes_kM^*))\\
&\cong H(R\Hom_{{\mathscr{A}}^e}({\mathscr{A}},N\otimes_kM^*)),
\end{align*}
where $(a)$ and $(b)$ are by Lemma \ref{imporcong} and Lemma
\ref{imporquasi} respectively.

\end{proof}

\begin{lem}\label{homsmooth}
Let $\mathscr{A}$ be a connected cochain DG algebra. Then ${}_{\mathscr{A}}k$ is compact if
and only if ${\mathscr{A}}$ is homologically smooth, which is also equivalent to
the condition that $k_{\mathscr{A}}$ is compact.
\end{lem}
\begin{proof}
Let $M = N = k$ in Lemma \ref{aelem}. Then we have
$$H(R\Hom_{{\mathscr{A}}^e}({\mathscr{A}},k))\cong H(R\Hom_{\mathscr{A}}(k,k)).$$
Hence $\dim_kH(R\Hom_{{\mathscr{A}}^e}({\mathscr{A}},k))<\infty$ if and only if $\dim_kH(R\Hom_{\mathscr{A}}(k,k))<\infty$.
By Lemma \ref{comp},  $\mathscr{A}$ is a
compact DG $\mathscr{A}^e$-module if and only if ${}_{\mathscr{A}}k$ is compact.
The DG module ${}_{\mathcal{A}}k$ is compact if and only if $k_{\mathscr{A}}$ is compact by considering the dimension of $H(k{}^L\otimes_{\mathscr{A}}k)$.
\end{proof}


\section {two invariants of dg modules}
The terminology `class'
in group theory is used to measure the shortest length of a
filtration with sub-quotients of certain type. Carlsson \cite{Car}
introduced `free class' for solvable free differential graded
modules over a graded polynomial ring. In \cite{ABI},
 Avramov, Buchweitz and Iyengar  introduced free class,
projective class and flat class for differential modules over a
commutative ring.
Inspired from them, the notion of DG free class for
semi-free DG modules was introduced in \cite{MW3}.
\begin{defn}\label{DGfclass}{\rm
 Let $F$ be a semi-free DG
$\mathscr{A}$-module. A semi-free filtration of $F$
$$0=F(-1)\subseteq F(0)\subseteq\cdots \subseteq F(n)\subseteq\cdots$$
is called strictly increasing, if $F(i-1) \neq F(i)$ when
$F(i-1)\neq F, i\ge 0$.  If there is some $n$ such that $F(n)=F$ and
$F(n-1)\neq F$, then we say that this strictly increasing semi-free
filtration has length $n$. If no such integer exists, then we say
the length is $+\infty$. The DG free class of $F$ is the shortest length of all strictly
increasing semi-free filtrations of $F$. We denote it
by $\mathrm{DG free\,\, class}_{\mathscr{A}}F$.}
\end{defn}

In general, it is hard to determine the DG free class of a semi-free
DG $\mathscr{A}$-module. For this, lets consider a special kind of semi-free
filtration. Let $F$ be a semi-free $A$-module with a semi-basis $E =
\{e_i|i\in I\}.$ Then $F^{\#}= \mathscr{A}^{\#}\otimes V$ is a free
$\mathscr{A}^{\#}$-module, where $V=\bigoplus\limits_{e\in E}ke$ is a
graded $k$-vector space spanned by $E$.
Let $V_0 = \{v\in V|\partial_F(v) = 0\}$ and define $F(0)$ as a DG
$\mathscr{A}$-submodule of $F$ with $F(0)^{\#}= \mathscr{A}^{\#}\otimes V_0$. Similarly,
let $V_{\leq 1} = \{v\in V|\partial_F(v)\in F(0) \}$, we define
$F(1)$ as a DG $\mathscr{A}$-submodule of $F$ such that $F(1)^{\#} = \mathscr{A}^{\#}
\otimes V_{\leq 1}$. It is easy to see that $F(1)$ is a semi-free
$\mathscr{A}$-submodule of $F$. Inductively, we suppose that $F(n)$ has been
defined. Let $V_{\leq n+1} = \{v \in V|
\partial_F(v) \in F(n) \}$ and define $F(n+1)$ as a DG $\mathscr{A}$-submodule of $F$
such that $F(n+1)^{\#}= \mathscr{A} \otimes V_{\leq n+1}$. In general, we let
$V(i)$ be a subspace of $V_{\leq i}$ such that $V_{\leq i}=V(i)
\oplus V_{\leq i-1}, i \ge 0.$ In this way, we define a strictly
increasing semi-free filtration of $F$:
$$0= F(-1)\subset F(0) \subset F(1) \subset\cdots \subset F(n) \subset \cdots,$$
such that each $F(i)/F(i-1)=\mathscr{A} \otimes V(i), i\ge 0$ is DG free on a
cocycle basis, which is also a basis of bi-graded $k$-vector space
$V(i)$.  Note that $\partial_F(v)\in F(i-1)$ but $\partial_F(v)\not\in F(i-2)$,
for any $v \in V(i)$. We call this semi-free filtration a standard
semi-free filtration of $F$ associated with the semi-basis $E$.
Obviously, the DG free class of $F$ must be equal to the length of
some standard semi-free filtration. In general, the lengths of standard semi-free filtrations of a
minimal semi-free DG $A$-module associated with different semi-basis
is generally not equal to each other. Lets see the following
example.

\begin{ex}\label{secex}
Let $\mathscr{A}$ be a connected cochain DG algebra such that there is a
graded element $a\in \mathscr{A}^1,$ $\partial_{\mathscr{A}}(a)=x \neq 0.$  Let $F$
be a semi-free $\mathscr{A}$-module such that  $$F^{\#}\cong
\bigoplus_{i=0}^n\mathscr{A}^{\#}e_i,$$ where the degree of $e_i$ is
$i$, and the differential is defined by
$$\partial_F(e_0) =0 \quad \text{and} \quad \partial_F(e_i)=\partial_{\mathscr{A}}(a)e_{i-1}-
a\partial_F(e_{i-1}), \,\, i\geq 1.$$ It is easy to check that $F$
has a standard semi-free filtration
$$0\subset F(0)\subset F(1) \subset\cdots \subset F(n-1) \subset F(n) =F,$$
such that $F(i)/F(i-1) = \mathscr{A}e_i, 1\le i\le n$. The length of this
filtration is $n$. On the other hand,
 $F = \mathscr{A}e_0\oplus \mathscr{A}(ae_0-e_1)\oplus \cdots \oplus
 \mathscr{A}(ae_{n-1}-e_n)$ is a DG free $\mathscr{A}$-module on a cocycle basis.
 Hence $F$ has a standard semi-free filtration of length $0$.
\end{ex}

In rational homotopy
theory, cone length of a topological space $X$ is defined to be the
least $m$ such that $X$ has the homotopy type of an $m$-cone. It is
a useful invariant in the evaluation of Lusternik-Schnirelmann
category, which is an important invariant of homotopy type. In \cite{MW3},
this invariant was introduced to DG homological algebra.
\begin{defn}\cite{MW3}{\rm
Let $M$ be a non-acyclic DG $\mathscr{A}$-module. The cone length of $M$ is defined to be
the number
$$\mathrm{cl}_AM =
\inf\{\mathrm{\,DGfree\,\,class}_{\mathscr{A}}F\,|\,F \stackrel{\simeq}\to M
 \ \text{is a semi-free resolution of}\  M\}.$$ And we define $\mathrm{cl}_{\mathscr{A}} N=-1$ if $H(N)=0$. }
\end{defn}

Note that $\mathrm{cl}_{\mathscr{A}}M$ may be $+\infty$. Cone length of a DG $\mathscr{A}$-module plays a similar role in DG
homological algebra as projective dimension of a module over a ring
does in classic homological ring theory. This invariant is called
`cone length' because any DG $\mathscr{A}$-module admits semi-free resolutions and the following lemma indicates that
semi-free DG $\mathscr{A}$-modules can be constructed by
iterative cone constructions from DG free $\mathscr{A}$-modules.

\begin{lem}\label{semicone}
Let $F$ be a semi-free DG $\mathscr{A}$-module and let $F'$ be a semi-free DG
submodule of $F$ such that $F/F'=\mathscr{A}\otimes V$ is DG free on a set of
cocycles. Then there exists a DG morphism $f: \mathscr{A}\otimes
\Sigma^{-1}V \to F'$ such that $F = \mathrm{cone}(f)$.
\end{lem}
\begin{proof}
Let $\{e_i|i\in I\}$ be a basis of $V$. We define DG morphism $f:
\mathscr{A}\otimes \Sigma^{-1}V \to F'$ by $f(\Sigma^{-1}e_i) =
\partial_F(e_i)$. It is easy to check that
$\partial_{\mathrm{cone}(f)}(e_i) = f(\Sigma^{-1}e_i) =
\partial_F(e_i)$. Hence $F = \mathrm{cone}(f).$
\end{proof}

\begin{prop}
Let $M$ be a DG $\mathscr{A}$-module with $\mathrm{cl}_{\mathscr{A}}M = 0$. If $M'$ is a
direct summand of $M$ such that $H(M')$ is bounded below, then $\mathrm{cl}_{\mathscr{A}}M' =0$.
\end{prop}
\begin{proof}
Since $\mathrm{cl}_{\mathscr{A}}M = 0$, $M$ admits a semi-free resolution $f:
F\stackrel{\simeq}{\to} M$ such that $\mathrm{DGfree\,\,class}_{\mathscr{A}}F
=0$. Clearly, $F$ is a DG free $\mathscr{A}$-module and is therefore minimal.

Since $H(M')$ is bounded below, the DG $\mathscr{A}$-module $M'$ has a minimal semi-free resolution $f':F'\to M'$.
Let $p: M \to M'$ and $i:M' \to M$ be the natural projection map and
the inclusion map respectively.

Since both $F$ and $F'$ are semi-free, there are DG morphisms $g:
F\to F'$ and $g': F' \to F$ such that $f'\circ g \sim p\circ f$ and
$f\circ g' \sim i \circ f'$. We have $f'\circ g\circ g' \sim p\circ
f\circ g' \sim p\circ i\circ f' = f'$. Since $f'$ is a
quasi-isomorphism, it is easy to check that $g\circ g'$ is a
quasi-isomorphism. By Lemma \ref{homotopyproj},
$$\Hom_{\mathscr{A}}(F',g\circ g'):\Hom_{\mathscr{A}}(F',F')\to \Hom_{\mathscr{A}}(F',F')$$
is a quasi-isomorphism. There exists $h\in Z^0(\Hom_{\mathscr{A}}(F',F'))$ such that $$\lceil \mathrm{id}_{F'}\rceil=H^0(\Hom_{\mathscr{A}}(F',g\circ g'))(\lceil h\rceil)= \lceil (g\circ g')\circ h\rceil.$$
Hence $(g\circ g')\circ h\simeq \mathrm{id}_{F'}$. Thus $h$ is also a quasi-isomorphism. By Lemma \ref{homotopyproj} again,
$$\Hom_{\mathscr{A}}(F',h):\Hom_{\mathscr{A}}(F',F')\to \Hom_{\mathscr{A}}(F',F')$$ is a quasi-isomorphism. There exists $q\in Z^0(\Hom_{\mathscr{A}}(F',F'))$ such that $$\lceil \mathrm{id}_{F'}\rceil=H^0(\Hom_{\mathscr{A}}(F',h))(\lceil q\rceil)= \lceil h\circ q\rceil.$$
So $h\circ q\sim \mathrm{id}_{F'}$. By Lemma \ref{homotopyinverse}, $h$ is a homotopy equivalence and $g\circ g'$ is a homotopy inverse of $h$.
Hence $g\circ g'$ is also a homotopy equivalence.
By Lemma \ref{homiso},
$g\circ g'$ is an isomorphism since $F'$ is minimal. This implies that $F'$ is a direct
summand of $F$. By Lemma \ref{dgfree},
$\mathrm{DGfree\,\,class}_{\mathscr{A}}F'=0$. Therefore $\mathrm{cl}_{\mathscr{A}}M' = 0$.
\end{proof}

\begin{prop}\label{finitecl}
Let $\mathscr{A}$ be a connected DG algebra such that
$\mathrm{cl}_{\mathscr{A}^e}\mathscr{A}< \infty$. Then for any DG $\mathscr{A}$-module $M$,  we
have $\mathrm{cl}_{\mathscr{A}}M\le \mathrm{cl}_{\mathscr{A}^e}\mathscr{A}$.
\end{prop}
\begin{proof}
Let $\mathrm{cl}_{\mathscr{A}^e}\mathscr{A}=n$. By the definition of cone length, the DG
$\mathscr{A}^e$-module $\mathscr{A}$ admits a semi-free resolution $X$ such that
$\mathrm{DGfree\,\,class}_{\mathscr{A}^e}X = n$. This implies that $X$ admits
a strictly increasing semi-free filtration
$$0=X(-1)\subset X(0)\subset X(1)\subset\cdots \subset X(n)=X,$$
where $X(0)= \mathscr{A}^e\otimes V(0)$ and $X(i)/X(i-1)\cong \mathscr{A}^e\otimes
V(i)$ is a DG free $\mathscr{A}^e$-module, $i=0,\cdots,n$. Let $E_i=\{e_{i_j}|j\in
I_i\}, i \ge 0,$ be a basis of $V(i)$. For any $i\ge 1$, define
$f_i: \mathscr{A}^e\otimes \Sigma^{-1}V(i)\to X(i-1)$ such that
$f_i(\Sigma^{-1}e_{i_j}) =
\partial_{X(i)}(e_{i_j})$. By Lemma \ref{semicone}, $X(i) \cong
\mathrm{cone}(f_i), i=1, 2,\cdots, n$.

For any DG $\mathscr{A}$-module $M$, let $\varrho_M:F\to M$ be a semi-free resolution of $M$. As
a DG $\mathscr{A}$-module, $X(i)\otimes_{\mathscr{A}}F\cong \mathrm{cone}(f_i\otimes_{\mathscr{A}}
\mathrm{id}_ {F})$, $i=1, 2,\cdots, n$.
 Since $\mathscr{A}^e\otimes_{\mathscr{A}}F\cong \mathscr{A}\otimes F$, we have
$$(\mathscr{A}^e\otimes V(i))\otimes_{\mathscr{A}}F\cong \mathscr{A}\otimes V(i)\otimes F,
\quad i= 0,1,\cdots, n.$$ Choose a subset $Z\subseteq F$ such
that each element $z\in Z$ is a cocycle and $\{\lceil z\rceil|z\in Z \}$ is a basis of the
$k$-vector space $H(F)$. Define a DG morphism
$$\phi_i: \mathscr{A}\otimes V(i)\otimes H(F)\to
\mathscr{A}\otimes V(i)\otimes F$$ such that $\phi_i(a \otimes v \otimes
\lceil z\rceil)=a \otimes v \otimes z$, for any $a \in \mathscr{A}, v \in V(i)$ and
$\lceil z\rceil$. It is easy to check that $\phi_i$ is a quasi-isomorphism.

In the following, we prove inductively that $\mathrm{cl}_{\mathscr{A}}
(X(i)\otimes_{\mathscr{A}}F)\le i, i= 0,1,\cdots, n$. Since $\phi_0: \mathscr{A}\otimes
V(0)\otimes H(F)\to X(0)\otimes_{\mathscr{A}}F$ is a quasi-isomorphism, we
have $\mathrm{cl}_{\mathscr{A}} (X(0)\otimes_{\mathscr{A}} F) = 0$. Suppose inductively
that we have proved that $$\mathrm{cl}_{\mathscr{A}}(X(l)\otimes_{\mathscr{A}}F)\le l,\,
l\ge 0.$$ We should prove $\mathrm{cl}_{\mathscr{A}}(X(l+1)\otimes_{\mathscr{A}}F)\le
l+1$. Since $\mathrm{cl}_{\mathscr{A}}(X(l)\otimes_{\mathscr{A}}F)\le l$, there is a
semi-free resolution $\varphi_l:G_l\stackrel{\simeq}{\to}
X(l)\otimes_{\mathscr{A}}F$ such that $\mathrm{DGfree\,\,class}_{\mathscr{A}}G_l\le l$.
 Because $\mathscr{A}\otimes\Sigma^{-1}V(l+1)\otimes H(F)$ is semi-free,
there is a DG morphism
$$\psi_{l}:\mathscr{A}\otimes\Sigma^{-1}V(l+1)\otimes H(F)\to
G_l$$ such that $\varphi_l\circ \psi_l \sim
(f_l\otimes_{\mathscr{A}}\mathrm{id}_{F})\circ \Sigma^{-1}(\phi_{l+1})$.

 For convenience, we write
$Q(l+1)=\mathscr{A}\otimes V(l+1)$ and $K(l+1) =\mathscr{A}^e\otimes V(l+1)$. In
$\mathscr{D}(\mathscr{A})$, there is a morphism
$h_{l+1}:\mathrm{cone}(\psi_l)\to X(l+1)\otimes_{\mathscr{A}}F$ making the
diagram
\begin{tiny}
\begin{align*}
\xymatrix{\Sigma^{-1}Q(l+1)\otimes H(F)
\ar[r]^{\psi_l}\ar[d]^{\Sigma^{-1}(\phi_{l+1})} & G_l
\ar[d]^{\varphi_l}\ar[r]^{\tau_{l}}&\mathrm{cone}(\psi_l)\ar[d]^{\exists
h_{l+1}}\ar[r]^{\varepsilon_l}&Q(l+1)\otimes_{k}H(F)\ar[d]^{\phi_{l+1}}
\\
\Sigma^{-1}K(l+1)\otimes_{\mathscr{A}}F \ar[r]^{f_l\otimes_{\mathscr{A}}\mathrm{id}_{F}}
&X(l)\otimes_{\mathscr{A}}F \ar[r]^{\iota_l}&X(l+1)\otimes_{\mathscr{A}}F
\ar[r]^{\pi_l}& K(l+1)\otimes_{\mathscr{A}}F
\\}
\end{align*}
\end{tiny}
commute. By five-lemma, $h_{l+1}$ is an isomorphism in
$\mathscr{D}(\mathscr{A})$. This implies that there are quasi-isomorphisms
$g:Y\to \mathrm{cone}(\psi_l)$ and $t:Y\to X(l+1)\otimes_{\mathscr{A}}F$,
where $Y$ is some DG $\mathscr{A}$-module. Hence $ \mathrm{cl}_{\mathscr{A}}
(X(l+1)\otimes_{\mathscr{A}}F)=\mathrm{cl}_{\mathscr{A}}Y =
\mathrm{cl}_{\mathscr{A}}\mathrm{cone}(\psi_l)\le l+1 $.
By induction, we have $\mathrm{cl}_{\mathscr{A}}(X\otimes_{\mathscr{A}}F)\le n$.
Since $F\simeq X\otimes_{\mathscr{A}}F$, we get $\mathrm{cl}_{\mathscr{A}}M\le n$.
\end{proof}

\begin{thm}\label{mindggf}
Let $M$ be an object in $\mathscr{D}^{+}(\mathscr{A})$ such that $\mathrm{cl}_{\mathscr{A}}M<\infty$, then there is a minimal semi-free resolution $G$ of $M$ such that $\mathrm{DGfree.class}_{\mathscr{A}}G=\mathrm{cl}_{\mathscr{A}}M$.
\end{thm}
\begin{proof}
Let $\mathrm{cl}_{\mathscr{A}}M=t$ and $b=\inf\{i|H^i(M)\neq 0\}$. There exists a semi-free resolution $P$ of $M$ such that $\mathrm{DGfree.class}_{\mathscr{A}}P=t$.
By \cite[Proposition 2.4]{MW1}, $M$ admits a minimal semi-free resolution $G$ with $G^{\#}\cong
\coprod_{i\ge b}\Sigma^{-i}(\mathscr{A}^{\#})^{(\Lambda^i)}$
each $\Lambda^i$ is an index set.
We have $P\cong G\oplus Q$ by Lemma \ref{semidecomp}, where $Q$ is a homotopically trivial DG $\mathscr{A}$-module.
 Set $F=G\oplus Q$. Then $\mathrm{DGfree\,\, class}_{\mathscr{A}}F =t$ and hence
 $F$ admits a semi-free filtration
$$0=F(-1)\subset F(0)\subset F(1)\subset \cdots \subset F(t)=F.$$
Let $E=\bigsqcup\limits_{i=0}^t E_i$ be a semi-basis of $F$ with respect to the semi-free filtration above.
For any $i\in \{1,2,\cdots, t\}$, let $E_i=\{e_{i_j}|j\in I_i \}$ and $F_i=F(i)/F(i-1)$. We have
 $\partial_F(e_{i_j})\subseteq \mathscr{A}(\bigsqcup\limits_{j=0}^{i-1}E_j)$.
Then each graded free $\mathscr{A}^{\#}$-module $F(r)^{\#}$ can be decomposed as
$$\bigoplus\limits_{i=0}^r\bigoplus\limits_{j\in I_i}\mathscr{A}^{\#}e_{i_j}, 0\le r\le t.$$
Let $e_{i_j}=g_{i_j}+q_{i_j}$, where $g_{i_j}\in G$ and $q_{i_j}\in Q$, for any $j\in I_i, i=0,1,\cdots, t$. We have
$$F(r)=[\sum\limits_{i=0}^r(\sum\limits_{j\in I_i}\mathscr{A}g_{i_j})]\oplus [\sum\limits_{i=0}^r(\sum\limits_{j\in I_i}\mathscr{A}q_{i_j})], \,\,\, 0\le r\le t.$$
Hence
$$F_r=F(r)/F(r-1)=(\sum\limits_{j\in I_r}\mathscr{A}\overline{g_{r_j}})\oplus (\sum\limits_{j\in I_r}\mathscr{A}\overline{q_{r_j}}),\,\,\, 0\le r\le t.$$
By Lemma \ref{dgfree}, $\sum\limits_{j\in I_r}\mathscr{A}\overline{g_{r_j}}$ is either a zero module or a DG free $\mathscr{A}$-module, for any $r=0,1,\cdots, t$.
Let $\omega_{r_{\lambda}}, \lambda\in \Lambda_r$ be its DG free basis ($\Lambda_r=\emptyset$ and $\omega_{r_{\lambda}}=0$ if $\sum\limits_{j\in I_r}\mathscr{A}\overline{g_{r_j}}=0$). Then $\sum\limits_{j\in I_r}\mathscr{A}\overline{g_{r_j}}=\bigoplus\limits_{j\in \Lambda_r}\mathscr{A}\omega_{r_j}$.
Note that $$\bigoplus\limits_{i\in I_r} \mathscr{A}^{\#}e_{r_i}=\bigoplus\limits_{\lambda\in \Lambda_r}\mathscr{A}^{\#}\omega_{r_{\lambda}}\oplus (\sum\limits_{j\in I_r}\mathscr{A}\overline{q_{r_j}})^{\#}$$
is a graded $\mathscr{A}^{\#}$-submodule of $F^{\#}$. So $\bigoplus\limits_{\lambda\in \Lambda_r}\mathscr{A}^{\#}\omega_{r_{\lambda}}$ is also a graded $\mathscr{A}^{\#}$-submodule of $F^{\#}$. Since $\partial_{F_r}(\omega_{r_{\lambda}})=0$, we have $$\partial_F(\omega_{r_{\lambda}})\in F(r-1)\cap \sum\limits_{i=0}^r(\sum\limits_{j\in I_i}\mathscr{A}g_{i_j})=\sum\limits_{i=0}^{r-1}(\sum\limits_{j\in I_i}\mathscr{A}g_{i_j}), \,\, 0\le r \le t.$$
Let $G(r)=\sum\limits_{i=0}^r(\sum\limits_{j\in I_i}\mathscr{A}g_{i_j}), r=0,1,\cdots, t$. Then
$$0\subseteq G(0)\subseteq G(1)\subseteq G(2)\subseteq \cdots \subseteq G(t)=G$$ is a filtration of DG $\mathscr{A}$-submodules of $G$.
Moreover, $G(r)/G(r-1)=\sum\limits_{j\in I_r}\mathscr{A}\overline{g_{r_j}}$ is either zero or a DG free $\mathscr{A}$-module
$\bigoplus\limits_{j\in \Lambda_r}\mathscr{A}\omega_{r_j}.$ If $G(r)/G(r-1)=0$, for some $r\in \{0,1,\cdots, t\}$, then we just cancel such $G(r)$. In this way, we can get a strictly increasing semi-free filtration with length smaller than $t$. Then $\mathrm{DGfree.class}_{\mathscr{A}}G<t=\mathrm{cl}_{\mathscr{A}}M$. It contradicts with $$\mathrm{cl}_{\mathscr{A}}M=\inf\{\mathrm{DGfree.class}_{\mathscr{A}}P| P\stackrel{\sim}{\to} M \,\text{is a semi-free resolution}\}\le \mathrm{DGfree.class}_{\mathscr{A}}G.$$
Therefore, $$0\subset G(0)\subset G(1)\subset G(2)\subset \cdots \subset G(t)=G$$
is a strictly increasing semi-free filtration of $G$. Then $\mathrm{DGfree.class}_{\mathscr{A}}G\le t$. On the other hand,
$t=\mathrm{cl}_{\mathscr{A}}M\le \mathrm{DGfree.class}_{\mathscr{A}}G$. Hence $ \mathrm{DGfree.class}_{\mathscr{A}}G=t$.
\end{proof}

In ring theory and homological algebra, it is well known that  the
global dimension of a ring $R$ is defined to be the supremum of the
set of projective dimensions of all $R$-modules. Since the invariant
cone length of a DG $\mathscr{A}$-module plays a similar role in DG
homological algebra as projective dimension of a module over a ring
does in homological ring theory, the following definition is reasonable to some extent.
\begin{defn} \cite{MW3}\label{defgldim} {\rm Let $\mathscr{A}$ be a connected cochain DG algebra. The left global
dimension and the right global dimension of $\mathscr{A}$ are respectively
defined by
$$l.\mathrm{Gl.dim}\mathscr{A} = \sup\{\mathrm{cl}_{\mathscr{A}}M| M\in \mathscr{D}(\mathscr{A})\}$$
and
$$r.\mathrm{Gl.dim}\mathscr{A} =
\sup\{\mathrm{cl}_{\mathscr{A}\!^{op}}M| M\in \mathscr{D}(\mathscr{A}\!^{op})\}.$$}
\end{defn}
Let $\mathscr{A}$ be a connected cochain DG algebra such that $H(\mathscr{A})$ is a graded algebra with finite global dimension.
 Then by the existence of Eilenberg-Moore resolution, one sees that any DG $\mathscr{A}$-module admits a semi-free resolution whose DG free class is not bigger than $\mathrm{gl.dim}H(\mathscr{A})$. So $l.\mathrm{Gl.dim}\mathscr{A} \le
\mathrm{gl.dim}H(\mathscr{A})$. If we assume in addition that $H(\mathscr{A})$ is Noetherian, then any cohomologically finitely generated DG $\mathscr{A}$-module is compact. Especially, the DG algebra $\mathscr{A}$ is homologically smooth by Lemma \ref{homsmooth}.  We emphasize that there are homologically smooth connected cochain DG algebras whose cohomology graded algebras are Noetherian graded algebras with infinite global dimension (see \cite[Example 3.12]{MW2}). So the homologically smoothness of $\mathscr{A}$ is weaker than $\mathrm{gl.dim}H(\mathscr{A})<\infty$ when $H(\mathscr{A})$ is Noetherian. Beside these, we have the following interesting results.

\begin{rem}\label{glineq}
We can similarly prove the following results as in \cite{MW3}.
\begin{enumerate}
\item $l.\mathrm{Gl.dim}\mathscr{A} =0$ if and only if $H(\mathscr{A})\simeq k$.
\item If $\partial_{\mathscr{A}}=0$, then $l.\mathrm{Gl.dim}\mathscr{A} =
\mathrm{gl.dim}\mathscr{A}^{\#}=r.\mathrm{Gl.dim}\mathscr{A}.$
\item $\mathrm{cl}_{\mathscr{A}}k =1$ if and only if
$l.\mathrm{Gl.dim}\mathscr{A} =1$ if and only if $\mathrm{gl.dim}H(\mathscr{A})=1.$
\item If $\mathrm{gl.dim}H(\mathscr{A})=2$, then
$l.\mathrm{Gl.dim}\mathscr{A} = \mathrm{cl}_{\mathscr{A}}k=2.$
\item If either $\mathrm{cl}_{\mathscr{A}}k $ or
$\mathrm{gl.dim}H(\mathscr{A})$ is finite and equals to
$\mathrm{depth}_{H(\mathscr{A})}H(\mathscr{A})$, then
$$l.\mathrm{Gl.dim}\mathscr{A} = \mathrm{gl.dim}H(\mathscr{A}) = \mathrm{cl}_{\mathscr{A}}k.$$
\end{enumerate}
Note that the DG algebras considered in \cite{MW3} are Adams connected DG algebras, which are a family of bigraded algebras. Although the DG algebras studied here are different from those in \cite{MW3}, the original proofs of the results above in \cite{MW3} are suitable to connected cochain DG algebras. The reason for this is because these two kinds of DG algebras admit unique maximal DG ideals and their underlying graded algebras are essentially connected graded algebras.
\end{rem}

\section{some criteria of homologically smooth dg algebras}
In DG homological algebra, homologically smooth DG algebras are fundamental and important as regular algebras in homological ring theory.
The motivation of this section is to figure out some criteria for a connected cochain DG algebra to be homologically smooth. The following
proposition will be useful for this purpose.
\begin{prop}\label{fgtocom}
Let $\mathscr{A}$ be a connected cochain DG algebra such that $H(\mathscr{A})$ is a Noetherian graded algebra. If $G$ is a minimal semi-free DG $\mathscr{A}$-module with finite DG free class in $\mathscr{D}_{fg}(\mathscr{A})$, then $G \in \mathscr{D}^c(\mathscr{A})$.
\end{prop}
\begin{proof}

 Let $\mathrm{DGfree\,\,class}_{\mathscr{A}}G=t<\infty$. Then $G$ admits a semi-free filtration
$$0=G(-1)\subset G(0)\subset G(1)\subset \cdots \subset G(t)=G$$
such that
 $G(i)/G(i-1) = \mathscr{A}\otimes W_i$ is
a DG free $\mathscr{A}$-module on a cocycle basis, for any $i\in \{0,1,\cdots, t\}$. It suffices to show  each
$\dim_kW_i<\infty$. Let $\{e_{i,j}|j\in I_i\}$ be a basis of
$W_i, i=0,1, \cdots, t$. Let $\iota_0: G(0)\to G$ be the
inclusion morphism. Since $\mathrm{im}H(\iota_0)$ is a graded
$H(\mathscr{A})$-submodule of $H(G)$ and $H(\mathscr{A})$ is a Noetherian graded algebra, we can conclude that
$\mathrm{im}H(\iota_0)\cong
\frac{H(G(0))}{\mathrm{ker}H(\iota_0)}$ is a finitely
generated $H(\mathscr{A})$-module.  Let $$\mathrm{im}H(\iota_0) =
H(\mathscr{A})f_{0,1}+H(\mathscr{A})f_{0,2}+\cdots+H(\mathscr{A})f_{0,n}.$$ Since $H(G(0))\cong
\bigoplus\limits_{j\in I_0}H(\mathscr{A})e_{0,j}$ is a free graded
$H(\mathscr{A})$-module, there is a finite subset $J_0=\{i_1,i_2,\cdots,i_l\}$
of $I_0$ such that
$$f_{0,s}=\sum\limits_{r=1}^la_{s,r}\overline{e_{0,i_r}}, s=1,2,
\cdots, n,$$ where each $a_{s,r}\in H(\mathscr{A})$. If $V(0)$ is infinite
dimensional, then both $I_0$ and $I_0\setminus J_0$ are infinite
sets. Hence for any $j\in I_0\setminus J_0$, we have $e_{0,j}\in
\mathrm{ker}H(\iota_0)$. Since $[\iota_{0}(e_{0,j})]=[e_{0,j}]=0$ in
$H(G)$, there exists $x_{0,j}\in G$ such that
$\partial_{G}(x_{0,j})=e_{0,j}$. This contradicts with the
minimality of $G$. Thus $W_0$ is finite dimensional and
$G(0)\in \mathscr{D}^f(\mathscr{A})$.

Assume inductively that $\dim_k W_j<\infty$ has been
proved $j=0,1,\cdots, i-1$. Then each $G(j)/G(j-1)$ is an object in $\mathscr{D}_{fg}(\mathscr{A})$, $j=0,1,\cdots,
i-1$.
We can prove inductively that
each $G(j)$ is in $\mathscr{D}_{fg}(\mathscr{A})$ by the
following sequence of short exact sequences
\begin{align*}
0\longrightarrow G(j-1)\longrightarrow G(j)\longrightarrow
G(j)/G(j-1)\longrightarrow 0,\,\, j= 1,\cdots, i-1.\\
\end{align*}
Similarly, $G/G(i-1)$ is also an object in $\mathscr{D}_{fg}(\mathscr{A})$ by
the short exact sequence $$0\longrightarrow G(i-1) \longrightarrow
G \longrightarrow G/G(i-1)\longrightarrow 0.$$
 On the other
hand, it is easy to see that $G/G(i-1)$ is also a minimal
semi-free DG $\mathscr{A}$-module and it has a semi-free filtration
$$G(i)/G(i-1)\subseteq G(i+1)/G(i-1)\subseteq \cdots \subseteq
G(t)/G(i-1)=G/G(i-1).$$ Let $\iota_i: G(i)/G(i-1)\to
G/G(i-1)$ be the inclusion morphism. Since
$\mathrm{im}H(\iota_i)$ is a graded $H(\mathscr{A})$-submodule of
$H(G/G(i-1))$ and $H(\mathscr{A})$ is Noetherian, one sees that
$\mathrm{im}H(\iota_i)\cong
\frac{H(G(i)/G(i-1))}{\mathrm{ker}H(\iota_i)}$ is a
finitely generated $H(\mathscr{A})$-module.  Let $$\mathrm{im}H(\iota_i) =
H(\mathscr{A})f_{i,1}+H(\mathscr{A})f_{i,2}+\cdots+H(\mathscr{A})f_{i,m}.$$ Since
$$H(G(i)/G(i-1))\cong \bigoplus\limits_{j\in I_i}H(\mathscr{A})e_{i,j}$$ is a
free graded $H(\mathscr{A})$-module, there is a finite subset
$J_i=\{s_1,s_2,\cdots,s_q\}$ of $I_i$ such that
$$f_{i,l}=\sum\limits_{j=1}^qa_{l,j}\overline{e_{i,s_j}}, l=1,2,
\cdots, m,$$ where each $a_{l,j}\in H(\mathscr{A})$. If $W_i$ is an infinite
dimensional space, then both $I_i$ and $I_i\setminus J_i$ are
infinite sets. Hence for any $j\in I_i\setminus J_i$, we have
$e_{i,j}\in \mathrm{ker}H(\iota_i)$. Since
$[\iota_{i}(e_{i,j})]=[e_{i,j}]=0$ in $H(G/G(i-1))$, there exist
$x_{i,j}\in G/G(i-1)$ such that
$\partial_{G}(x_{i,j})=e_{i,j}$. This contradict with the
minimality of $G$. Thus $W_i$ is finite dimensional.

By the induction above, we get $\dim_k W_i<\infty$ for any $i\in \{0,1,\cdots,t\}$. Hence $G$ has a finite semi-basis and $G$
is compact.
\end{proof}

The following theorem  completely characterize homologically smooth DG algebra intrinsically.

\begin{thm}\label{mainres}
Let $\mathscr{A}$ be a connected cochain DG algebra such that $H(\mathscr{A})$ is a Noetherian graded algebra. Then the following statements are equivalent:

$(a)\,\, \mathscr{A}$ is homologically smooth.

$(b)\,\, \mathrm{cl}_{\mathscr{A}^e}\mathscr{A}<\infty$.

$(c)\,\, l.\mathrm{Gl.dim}\,\mathscr{A}<\infty$.

$(d)\,\, \mathscr{D}^c(\mathscr{A})=\mathscr{D}_{fg}(\mathscr{A}) $.

$(e)\,\, \mathscr{D}_{sg}(\mathscr{A})=0$.

$(f)\,\, \mathrm{cl}_{\mathscr{A}}k<\infty$.

$(g)\,\, k\in \mathscr{D}^c(\mathscr{A})$.

\end{thm}

\begin{proof}
(a)$\Rightarrow$(b) Since $\mathscr{A}$ is a homologically smooth DG algebra, the DG $\mathscr{A}^e$-module $\mathscr{A}$ is compact. So it admits a minimal semi-free resolution $X$ with a finite semi-basis. This implies that $\mathrm{DGfree.class}_{\mathscr{A}^e}X<\infty$.
By the definition of cone length, $$\mathrm{cl}_{\mathscr{A}^e}\le \mathrm{DGfree.class}_{\mathscr{A}^e}X<\infty.$$

(b)$\Rightarrow$(c) For any DG $\mathscr{A}$-module $M$, we have $\mathrm{cl}_{\mathscr{A}}M\le \mathrm{cl}_{\mathscr{A}^e}\mathscr{A}$
by Proposition \ref{finitecl}.
Therefore, $l.\mathrm{Gl.dim}\,\mathscr{A}=\sup\{\mathrm{cl}_{\mathscr{A}}M| M\in \mathscr{D}(\mathscr{A})\}\le \mathrm{cl}_{\mathscr{A}^e}\mathscr{A}<\infty$.

(b)$\Rightarrow$(d) It suffices to show that any DG $\mathscr{A}$-module $M$ in $\mathscr{D}_{fg}(\mathscr{A})$ is compact.
 By Proposition \ref{finitecl}, we have $\mathrm{cl}_{\mathscr{A}}M\le \mathrm{cl}_{\mathscr{A}^e}\mathscr{A} <\infty$. By Proposition \ref{mindggf}, $M$ admits a minimal semi-free resolution $G$ such that $\mathrm{DGfree. class}_{\mathscr{A}}G=\mathrm{cl}_{\mathscr{A}}M$.
 Then $G$ is an object in $\mathscr{D}^c(\mathscr{A})$ by Proposition \ref{fgtocom}.  Since $G$ is a semi-free resolution of $M$, we conclude that $M\in \mathscr{D}^c(\mathscr{A})$.

 (d)$\Leftrightarrow$(e) Since $\mathscr{D}_{sg}(\mathscr{A})=\mathscr{D}_{fg}(\mathscr{A})/\mathscr{D}^c(\mathscr{A})$, $\mathscr{D}_{fg}(\mathscr{A})=\mathscr{D}^c(\mathscr{A}) \Leftrightarrow \mathscr{D}_{sg}(\mathscr{A})=0$.

 (c)$\Rightarrow$ (f) We have $\mathrm{cl}_{\mathscr{A}}k\le \sup\{\mathrm{cl}_{\mathscr{A}}M| M\in \mathscr{D}(\mathscr{A})\}=l.\mathrm{Gl.dim}\,\mathscr{A}<\infty$.

 (d)$\Rightarrow$(f) Since $k\in \mathscr{D}_{fg}(\mathscr{A})=\mathscr{D}^c(\mathscr{A})$, $k$ admits a minimal semi-free resolution $F_k$ which has a finite semi-basis. We have $\mathrm{DGfree.class}_{\mathscr{A}}F_k<\infty$. Therefore, $$\mathrm{cl}_{\mathscr{A}}k\le \mathrm{DGfree.class}_{\mathscr{A}}F_k<\infty.$$

 (f)$\Rightarrow$(g) Let $\mathrm{cl}_{\mathscr{A}}k=t$.  By Proposition \ref{mindggf}, $k$ has a minimal semi-free resolution $F_k$ such that $\mathrm{DGfree.class}_{\mathscr{A}}F_k=t$. Applying Proposition \ref{fgtocom} to $F_k$, we conclude $F_k\in \mathscr{D}^c(\mathscr{A})$.  Then $k\in \mathscr{D}^c(\mathscr{A})$ since $F_k$ is a semi-free resolution of ${}_{\mathscr{A}}k$.

 (g)$\Rightarrow$(a) By Lemma\ref{homsmooth}, $\mathscr{A}$ is homologically smooth since  $k\in \mathscr{D}^c(\mathscr{A})$.
\end{proof}

By \cite[Proposition 4.6]{MW2}, we have quasi-inverse
contravariant equivalences of categories,
\begin{align*}
\xymatrix{&\mathscr{D}^{c}(\mathscr{A})\quad\quad\ar@<1ex>[r]^{R\Hom_{\mathscr{A}}(-,
\mathscr{A})}&\quad\quad
\mathscr{D}^{c}(\mathscr{A}\!^{op})\ar@<1ex>[l]^{R\Hom_{\mathscr{A}\!^{op}}(-,\mathscr{A})}}.
\end{align*}
By Theorem \ref{mainres} and Lemma \ref{homsmooth}, $\mathscr{D}^{c}(\mathscr{A})=\mathscr{D}_{fg}(\mathscr{A})$ and $\mathscr{D}^{c}(\mathscr{A}^{op})=\mathscr{D}_{fg}(\mathscr{A}^{op})$ when $\mathscr{A}$ is homologically smooth and $H(\mathscr{A})$ is Noetherian. The following corollary is obviously true.

\begin{cor}\label{dual}
Let $\mathscr{A}$ be a homologically smooth connected cochain DG algebra such that $H(\mathscr{A})$ is a Noetherian graded algebra.
There is a duality between $\mathscr{D}_{fg}(\mathscr{A})$ and
$\mathscr{D}_{fg}(\mathscr{A}\!^{op})$. To be precise, we have quasi-inverse
contravariant equivalences of categories,
\begin{align*}
\xymatrix{&\mathscr{D}_{fg}(\mathscr{A})\quad\quad\ar@<1ex>[r]^{R\Hom_{\mathscr{A}}(-,
\mathscr{A})}&\quad\quad
\mathscr{D}_{fg}(\mathscr{A}\!^{op})\ar@<1ex>[l]^{R\Hom_{\mathscr{A}\!^{op}}(-,\mathscr{A})}}.
\end{align*}
\end{cor}

\section{Ext and Castelnuovo-Mumford regularities of DG modules}
In this section, we study the Ext and Castelnuovo-Mumford regularities of DG modules. These two invariants of DG modules were introduced and studied in \cite{Jor2}.
\begin{defn}\label{exreg}{\rm
 For any $M \in \mathscr{D}(\mathscr{A})$, we define the Ext-regularity of $M$ by
$$\mathrm{Ext.reg}\, M = -\inf\{i|H^i(R\Hom_{\mathscr{A}}(M,k))\neq 0\},$$ and similarly for $N \in \mathrm{D}(\mathscr{A}^{op})$.  Note that $\mathrm{Ext.reg}(0)=-\infty$.  }
\end{defn}

\begin{rem}
For any  DG $\mathscr{A}$-module $M$ in $\mathscr{D}_{fg}(\mathscr{A})$, it admits a minimal semi-free resolution $F_M$ by Lemma \ref{exist}. Let $E$ be a semi-basis of $F_M$. Then by the minimality of $F_M$, we have $$\mathrm{Ext.reg}\, M=\sup\{|e|\, |e\in E\}.$$ If $\mathscr{A}$ is homologically smooth and $H(\mathscr{A})$ is Noetherian, then $\mathscr{D}_{fg}(\mathscr{A})=\mathscr{D}^c(\mathscr{A})$ by Theorem \ref{mainres} and hence any object in $\mathscr{D}_{fg}(\mathscr{A})$ has finite Ext-regularity.
\end{rem}

\begin{defn}\cite{MW1}
{\rm For any object $M \in \mathscr{D}(\mathscr{A})$, the depth  and $k$-injective dimension of $M$ are defined, respectively, as
$$
\mathrm{depth}_{\mathscr{A}}M=\inf\{j|H^j(R\Hom_{\mathscr{A}}(k,M))\neq 0\}$$
and $$ k.\mathrm{id}_{\mathscr{A}}M =\sup\{j|H^j(R\Hom_{\mathscr{A}}(k,M))\neq 0\}.$$
}
\end{defn}

In the rest of this section,  we assume that $\mathscr{A}$ is a homologically smooth connected cochain DG algebra.
Then both ${}_{\mathscr{A}}k$ and $k_{\mathscr{A}}$ are compact by Lemma \ref{homsmooth}. In this case, we have \cite[Setup 4.1]{Jor2}.
 Let $K$ and $L$ be the minimal semi-free resolutions of ${}_{\mathscr{A}}k$ and $k_{\mathscr{A}}$, respectively.  We have $\langle K\rangle = \langle {}_{\mathscr{A}}k \rangle$ and $\langle L\rangle =\langle k_{\mathscr{A}}\rangle$  in $\mathscr{D}(\mathscr{A})$.
Set $$\mathcal{N}=\langle {}_{\mathscr{A}}k\rangle^{\bot} =\langle {}_{\mathscr{A}}K\rangle^{\bot}, \mathscr{D}^{\mathrm{tors}}(\mathscr{A})={}^{\bot}\mathcal{N}\,\, \text{and}\,\, \mathscr{D}^{\mathrm{comp}}(\mathscr{A})=\mathcal{N}^{\bot}$$ in $\mathscr{D}(\mathscr{A})$. The DG modules in $\mathscr{D}^{\mathrm{tors}}(\mathscr{A})$ and $\mathscr{D}^{\mathrm{comp}}(\mathscr{A})$ are called torsion DG modules and complete DG modules, respectively.
Then $\mathscr{D}^{\mathrm{tors}}(\mathscr{A})=\langle{}_{\mathscr{A}}k\rangle=\langle{}_{\mathscr{A}}K\rangle$.
Let $\mathcal{E}=\Hom_{\mathscr{A}}(K,K)$ be the endomorphism DG algebra.
We have the following lemma on $\mathcal{E}$.
 \begin{lem}\label{local}The DG algebra $\mathcal{E}$ satisfies the following conditions.
 \begin{enumerate}
\item $\dim_kH(\mathcal{E})<\infty$;
\item  $0=\sup\{i\in \Bbb{Z}|H^i(\mathcal{E})\neq 0\}$;
\item  $H^0(\mathcal{E})$ is a local finite dimensional algebra.
 \end{enumerate}
 \end{lem}
 \begin{proof}
(1) Since $\mathscr{A}$ is homologically smooth, the minimal semi-free resolution $K$ of ${}_{\mathscr{A}}k$ has a finite semi-basis $E$.
 By the minimality of $K$, one sees that \begin{align*}
 \dim_kH(\mathcal{E})&=\dim_kH(\Hom_{\mathscr{A}}(K,K))\\
 &=\dim_k\Hom_{\mathscr{A}}(K,k)\\
 &=\dim_k\oplus_{e\in E}ke=|E|<\infty.
 \end{align*}
 (2)By Lemma \ref{exist}, ${}_{\mathscr{A}}k$ has a minimal semi-free resolution $K$ such that
$$K^{\#}= \coprod\limits_{i\ge
0}\Sigma^{-i}(\mathscr{A}^{\#})^{(\Lambda^i)},$$ where each $\Lambda^i$ is an
index set. Thus \begin{align*}
\sup\{j\in \Bbb{Z}|H^j(\mathcal{E})\neq 0\}&=\sup\{i\in\Bbb{Z}|[\Hom_{\mathscr{A}}(K,k)]^i\neq 0\}\\
&=\sup\{j\in\Bbb{Z}| [\Hom_{\mathscr{A}}(\coprod\limits_{i\ge
0}\Sigma^{-i}(\mathscr{A}^{\#})^{(\Lambda^i)},k)]^j\neq 0\}\\
&=\sup\{j\in\Bbb{Z}| [\prod\limits_{i\ge
0}\Sigma^{i}(k)^{(\Lambda^i)}]^j\neq 0\}=0.
\end{align*}
(3) By \cite[Lemma 10.2]{MW2}, the algebra $\Hom_{\mathscr{D}(\mathscr{A})}(k,k)$ is local. Thus
the algebra $$H^0(\mathcal{E})=H^0(R\Hom_{\mathscr{A}}(k,k))=\Hom_{\mathscr{D}(\mathscr{A})}(k,k)$$ is a finite dimensional local algebra since $\dim_kH(\mathcal{E})<\infty$.
 \end{proof}

By Lemma \ref{local}, $H^0(\mathcal{E})$ is a local algebra. Let $J$ be its maximal ideal.
Set $b=\inf\{i|H^i(\mathcal{E})\neq 0$, $Z^i=\mathrm{ker}(d_{\mathcal{E}}^i)$, $C^i=\mathcal{E}^i/Z^i$, $H^i=H^i(\mathcal{E})$ and $B^i=\mathrm{im}(d_{\mathcal{E}}^{i-1})$. Then $\mathcal{E}$ admits two DG subalgebras
$$\mathcal{E}':\quad\quad \cdots \stackrel{d_{\mathcal{E}}^{i-1}}{\to} \mathcal{E}^i\stackrel{d_{\mathcal{E}}^{i}}{\to} \mathcal{E}^{i+1}
\stackrel{d_{\mathcal{E}}^{i+1}}{\to}\cdots \stackrel{d_{\mathcal{E}}^{-2}}{\to}\mathcal{E}^{-1}\stackrel{d_{\mathcal{E}}^{-1}}{\to} Z^0\to 0$$
and
$$\mathcal{E}'':\quad\quad  \cdots \stackrel{d_{\mathcal{E}}^{j-1}}{\to} \mathcal{E}^j\stackrel{d_{\mathcal{E}}^{j}}{\to} \mathcal{E}^{j+1}
\stackrel{d_{\mathcal{E}}^{j+1}}{\to}\cdots  \stackrel{d_{\mathcal{E}}^{b-1}}{\to} B^b\to 0.$$ Clearly, $\mathcal{E}''$ is a DG ideal of $\mathcal{E}'$. Note that the DG algebra $\mathcal{E}'/\mathcal{E}''$ is
$$0\to C^b\oplus H^b \stackrel{d_{\mathcal{E}}^{b}}{\to} \mathcal{E}^{b+1}\stackrel{d_{\mathcal{E}}^{b+1}}{\to} \cdots \stackrel{d_{\mathcal{E}}^{-2}}{\to} \mathcal{E}^{-1} \stackrel{d_{\mathcal{E}}^{-1}}{\to} Z^0\to 0.$$
One sees that both the inclusion morphism $\iota: \mathcal{E}'\to \mathcal{E}$ and the canonical surjection $\varepsilon: \mathcal{E}'\to \mathcal{E}'/\mathcal{E}''$ are quasi-isomorphisms.
  Let $R_t=(\mathcal{E}')^{-t}$ and $d_t^R=d_{\mathcal{E}}^{-t}$ for any $t\ge 0$.  In this way, $\mathcal{E}'$ can be considered as a chain DG algebra $R$:
$$\cdots\stackrel{d^R_{i+1}}{\to} R_{i}\stackrel{d_{i}^R}{\to} R_{i-1}\stackrel{d_{i-1}^R}{\to}\cdots \stackrel{d_{2}^R}{\to} R_1\stackrel{d_{1}^R}{\to} R_0\to 0.$$
Moreover,
$H_0(R)=R_0/\mathrm{im}(d_1^R)\cong H^0$ is a finite dimensional local algebra and $\dim_kH(R)=\dim_kH(\mathcal{E}')=\dim_k H(\mathcal{E})<\infty$. Each $H_i(R)$ is a finitely generated $H_0(R)$-module and $-b=\sup\{i\in \Bbb{Z}|H_i(R)\neq 0\}$.  So $R$ is  a local chain DG algebra introduced in \cite{FJ}. Its maximal DG ideal is
$$\mathfrak{m}_R: \quad \cdots \stackrel{d^R_{i+1}}{\to} R_i\stackrel{d_i^R}{\to}\cdots \stackrel{d_{2}^R}{\to}R_1 \stackrel{d_1^R}{\to} R_0=B^0\oplus J\to 0.$$
\begin{rem}
The DG algebra $\mathcal{E}'$ and $\mathcal{E}$ are both augmented DG algebras with augmented DG ideals
\begin{align*}
\mathfrak{m}_{\mathcal{E}'}: \quad\quad \cdots \stackrel{d_{\mathcal{E}}^{i-1}}{\to} \mathcal{E}^i\stackrel{d_{\mathcal{E}}^{i}}{\to}\cdots \stackrel{d_{\mathcal{E}}^{-2}}{\to}\mathcal{E}^{-1} \stackrel{d_{\mathcal{E}}^{-1}}{\to} B^0\oplus J\to 0
\end{align*}
and
\begin{align*}
\mathfrak{m}_{\mathcal{E}}: \quad\quad \cdots \stackrel{d_{\mathcal{E}}^{i-1}}{\to} \mathcal{E}^i\stackrel{d_{\mathcal{E}}^{i}}{\to}\cdots \stackrel{d_{\mathcal{E}}^{-2}}{\to}\mathcal{E}^{-1} \stackrel{d_{\mathcal{E}}^{-1}}{\to} B^0\oplus J\oplus C^0 \stackrel{d_{\mathcal{E}}^{0}}{\to} \mathcal{E}^1 \stackrel{d_{\mathcal{E}}^{1}}{\to} \cdots \stackrel{d_{\mathcal{E}}^{j}}{\to} \mathcal{E}^j\stackrel{d_{\mathcal{E}}^{j+1}}{\to}\cdots.
\end{align*}
\end{rem}

 \begin{prop}\label{minres}
 Let $X$ be a left DG $\mathcal{E}'$-module such that each $H^i(X)$ is a finitely generated $H^0(\mathcal{E}')$-module and $u=\sup\{i|H^i(X)\neq 0\}<\infty$. Then $X$ admits a minimal semi-free resolution $F$ with $F^{\#}_X=\coprod\limits_{j\le u}\Sigma^j(\mathcal{E}'^{\#})^{(\beta_j)}$, where each $\beta_j$ is finite.
 \end{prop}

\begin{proof}
Let $M=\bigoplus\limits_{j\in \Bbb{Z}} M_j$ with $M_{-i}=X^i$ for any $i\in \Bbb{Z}$. Then $M$ is DG $R$-module such that each $H_i(M)$ is a finitely generate $H_0(R)$-module. And $H(M)$ is bounded below with $-u=\inf\{i|H_i(M)\neq 0\}$. It follows from \cite[0.5]{FJ} that $M$ admits a minimal semi-free resolution $G$ such that
$$G^{\#}=\coprod\limits_{i\ge -u}\Sigma^i(R^{\#})^{(\beta_i)},$$ where each $\beta_i$ is finite. Let $F^i=G_{-i}$. Then $F$ is a minimal semi-free $\mathcal{E}'$-module with
$$F^{\#}=\coprod\limits_{j\le u}\Sigma^j(\mathcal{E}'^{\#})^{\beta_j}.$$ Moreover, it is a minimal semi-free resolution of $X$.
\end{proof}

\begin{prop}\label{semifree}
Let $N$ be a DG $\mathcal{E}$-module such that $u=\sup\{i|H^i(X)\neq 0\}<\infty$ and each $H^i(N)$ is a finitely generated $H^0(\mathcal{E})$-module.
Then $N$ admits a minimal semi-free resolution $F$ such that $F^{\#}=\coprod\limits_{j\le u}\Sigma^j(\mathcal{E}^{\#})^{\beta_j}$,  where each $\beta_j$ is finite.
\end{prop}

\begin{proof}
Via the inclusion morphism $\iota: \mathcal{E}'\to \mathcal{E}$, $N$ can be considered as a DG $\mathcal{E}'$-module. By Proposition \ref{minres}, ${}_{\mathcal{E}'}N$ admits a minimal semi-free resolution $G$ such that $$G^{\#}=\coprod\limits_{j\le u}\Sigma^j(\mathcal{E}'^{\#})^{\beta_j},$$  where each $\beta_j$ is finite. One sees easily that $F=\mathcal{E}\otimes_{\mathcal{E}'}G$ is a minimal semi-free resolution of ${}_{\mathcal{E}}N$ and
$$F^{\#}=\coprod\limits_{j\le u}\Sigma^j(\mathcal{E}^{\#})^{\beta_j},$$ where each $\beta_j$ is finite.
\end{proof}

 The DG module $K$ acquires the structure ${}_{\mathscr{A},{}_{\mathcal{E}}}K$ while $K^*=\Hom_{\mathscr{A}}(K,\mathscr{A})$ has the structure $K^*_{\mathscr{A},\mathcal{E}}$. Define functors \,\, $T(-)=-\,{}^L{\otimes}_{\mathcal{E}}K$,
$$W(-)=\Hom_{\mathscr{A}}(K,-)\simeq K^*\,{}^L{\otimes}_{\mathscr{A}}- \quad \text{and}\quad
C(-)=R\Hom_{\mathcal{E}^{op}}(K^*,-),$$
which form adjoint pairs $(T,W)$ and $(W,C)$ between $\mathscr{D}(\mathcal{E}^{op})$ and $\mathscr{D}(\mathscr{A})$.
There are pairs of quasi-inverse equivalences of categories as follows
$$\xymatrix{\mathscr{D}^{\mathrm{comp}}(\mathscr{A})   \ar@<1ex>[r]^{\quad W}
& \mathscr{D}(\mathcal{E}^{op}) \ar@<1ex>[l]^{\quad C}  \ar@<1ex>[r]^{T} & \mathscr{D}^{\mathrm{tors}}(\mathscr{A}) \ar@<1ex>[l]^{W} } .$$
In particular, $WC$ and $WT$ are equivalent to the identity functor on $\mathscr{D}(\mathcal{E}^{op})$ , so if we set $$\Gamma =TW, \Lambda=CW,$$
then we get endofunctors of $\mathscr{D}(\mathscr{A})$ which form an adjoint pair $(\Gamma, \Lambda)$ and satisfy
$$\Gamma^2\simeq \Gamma, \Lambda^2\simeq \Lambda, \Gamma\Lambda\simeq \Gamma, \Lambda\Gamma\simeq \Lambda.$$
These functors are adjoints of inclusions as follows, where left-adjoint are displayed above right-adjoints
$$\xymatrix{\mathscr{D}^{\mathrm{comp}}(\mathscr{A})   \ar@<1ex>[r]^{\quad \mathrm{inc}}
& \mathscr{D}(\mathscr{A}) \ar@<1ex>[l]^{\quad \Lambda}  \ar@<1ex>[r]^{\Gamma} & \mathscr{D}^{\mathrm{tors}}(\mathscr{A}) \ar@<1ex>[l]^{ \mathrm{inc }} } .$$
Write $Q=K^*{}^L\otimes_{\mathcal{E}}K$ and $D=Q^{\vee}=\Hom_{k}(Q,k)$. One sees that $Q$ and $D$ have the structures ${}_{\mathscr{A}}Q_{\mathscr{A}}$ and ${}_{\mathscr{A}}D_{\mathscr{A}}$,  respectively.  From the definitions, we have $$\Gamma(-)=Q\,{}^L{\otimes}_{\mathscr{A}}-\quad \text{and} \quad \Lambda(-)=R\Hom_{\mathscr{A}}(Q,-).$$
The following definition was introduced in \cite[Definition 5.1]{Jor2}.
\begin{defn}{\rm
For any DG $\mathscr{A}$-module $M$, its Castelnuovo-Mumford regularity is defined by
$$\mathrm{CMreg}M=\sup\{i|H^i(\Gamma (M))\neq 0\}.$$
Note that $ \mathrm{CMreg}(0)=-\infty$.}
\end{defn}
\begin{defn}\cite{FHT1}\label{gordef}{\rm
Let $\mathscr{A}$ be a connected cochain DG algebra.
If $$\dim_kH(R\Hom_{\mathscr{A}}(k,\mathscr{A})) =1,\,\,(\text{resp.} \dim_{k}H(R\Hom_{\mathscr{A}^{op}}(k,\mathscr{A}))=1),$$ then $\mathscr{A}$ is called left (resp. right) Gorenstein.   If $\mathscr{A}$ is both left Gorenstein and right Gorenstein, then we say that $\mathscr{A}$ is Gorenstein. }
\end{defn}
\begin{rem}
Assume that $\mathscr{A}$ is a left Gorenstein DG algebra.  Then we have
$k.\mathrm{id}_{\mathscr{A}}\mathscr{A}=\mathrm{depth}_{\mathscr{A}}\mathscr{A}$ since $\dim_kH(R\Hom_{\mathscr{A}}(k,\mathscr{A}))=1$. By the way, the invariant $k.\mathrm{id}_{\mathscr{A}}\mathscr{A}$ is called `formal dimension' of $\mathscr{A}$ in \cite{Gam}. Although a left Gorenstein DG algebra is not necessarily right Gorenstein in noncommutative setting. For any homologically smooth DG algebra $\mathscr{A}$, it is left Gorenstein if and only if it is right Gorenstein by \cite[Remark 7.6]{MW2}.
\end{rem}

\begin{prop}\label{eqfd}
 Suppose that $\mathscr{A}$ is a homologically smooth connected cochain DG algebra. If $\mathscr{A}$ is Gorenstein, then $\mathrm{depth}_{\mathscr{A}}\mathscr{A}=\mathrm{depth}_{\mathscr{A}^{op}}\mathscr{A}$.
\end{prop}
\begin{proof}
Let $\mathrm{depth}_{\mathscr{A}}\mathscr{A}=m$ and $\mathrm{depth}_{\mathscr{A}^{op}}\mathscr{A}=n$. Then $H(R\Hom_{\mathscr{A}}(k,\mathscr{A}))\cong \Sigma^{-m}k$ and
and $H(R\Hom_{\mathscr{A}^{op}}(k,\mathscr{A}))\cong \Sigma^{-n}k$. By \cite[Lemma 2.7]{MGYC}, $R\Hom_{\mathscr{A}}(k,\mathscr{A})\simeq \Sigma^{-m}k_{\mathscr{A}}$ in $\mathscr{D}(\mathscr{A}^{op})$ and $R\Hom_{\mathscr{A}^{op}}(k,\mathscr{A})\cong \Sigma^{-n}{}_{\mathscr{A}}k$ in $\mathscr{D}(\mathscr{A})$. Since ${}_{\mathscr{A}}k$ is compact, the biduality morphism
$$k\to R\Hom_{\mathscr{A}^{op}}(R\Hom_{\mathscr{A}}(k,\mathscr{A}),\mathscr{A})$$ is a quasi-isomorphism by \cite[Proposition 4.6]{MW2}. On the other hand,
\begin{align*}
 H(R\Hom_{\mathscr{A}^{op}}(R\Hom_{\mathscr{A}}(k,\mathscr{A}),\mathscr{A}))&\cong H(R\Hom_{\mathscr{A}^{op}}(\Sigma^{-m}k_{\mathscr{A}},\mathscr{A}))\\
                                                      &\cong \Sigma^{m-n} k.
\end{align*}
Thus $m=n$.
\end{proof}

\begin{thm}\label{formula}
Let $\mathscr{A}$ be a Gorenstein and homologically smooth connected cochain DG algebra such that $H(\mathscr{A})$ is a Noetherian graded algebra. Then for any object $M$ in $\mathscr{D}^f(\mathscr{A})$, we have
$$\mathrm{CMreg}M = \mathrm{depth}_{\mathscr{A}}\mathscr{A} + \mathrm{Ext.reg}\, M<\infty.$$
\end{thm}
\begin{proof}
By Theorem \ref{mainres}, we have $M\in \mathscr{D}^c(\mathscr{A})$. Then $M$ admits a minimal semi-free resolution $F$ with a finite semi-basis $E$.  By the minimality of $F$, $$H(R\Hom_{\mathscr{A}}(M,k))=\Hom_{\mathscr{A}}(F,k)\cong \bigoplus\limits_{e\in E}ke.$$
One sees clearly that $$\mathrm{Ext.reg}\, M= -\inf\{i|H^i(R\Hom_{\mathscr{A}}(M,k))\neq 0\}=\sup\{|e|\,\, |\,\,e\in E\}<\infty.$$
Let $b=\inf\{i\in\Bbb{Z}|H^i(M)\neq 0\}$, $u=\mathrm{Ext.reg}\, M$ and $t=\mathrm{depth}_{\mathscr{A}}\mathscr{A}$. Then $$K^*=\Hom_{\mathscr{A}}(K,\mathscr{A})\cong \Sigma^{-t}k_{\varepsilon,\mathscr{A}}$$ in $\mathscr{D}(\mathscr{A}^{op})$ and $$F^{\#}=\coprod\limits_{b\le j\le u}\Sigma^{-j}(\mathscr{A}^{\#})^{(\alpha_j)},$$
where each $\alpha_j$ is finite. By Proposition \ref{semifree}, ${}_{\mathscr{E}}K$ admits a minimal semi-free resolution $P$ such that $P^{\#}=\coprod\limits_{j\le 0}\Sigma^j(\mathcal{E}^{\#})^{(\beta_j)}$,  where each $\beta_j$ is finite.
Therefore, \begin{align*}
\mathrm{CMreg}M &= \sup\{i\in \Bbb{Z}|H^i(\Gamma (M))\neq 0\}\\
&= \sup\{i\in \Bbb{Z}|H^i[(K^*{}^L\otimes_{\mathcal{E}}K)\otimes_{\mathscr{A}}F]\neq 0\}\\
&=\sup\{i\in \Bbb{Z}|H^i[(\Sigma^{-t}k_{\varepsilon,\mathscr{A}}\otimes_{\mathcal{E}}P)\otimes_{\mathscr{A}}F]\neq 0\}\\
&=\sup\{i\in \Bbb{Z}|H^i[(\Sigma^{-t}k_{\varepsilon,\mathscr{A}}\otimes_{\mathcal{E}}\coprod\limits_{j\le 0}\Sigma^j(\mathcal{E}^{\#})^{(\beta_j)} )\otimes_{\mathscr{A}}F]\neq 0\} \\
&=\sup\{i\in \Bbb{Z}|[(\coprod\limits_{j\le 0}\Sigma^{-t+j}k^{(\beta_j)})\otimes_{\mathscr{A}}\coprod\limits_{b\le q\le u}\Sigma^{-q}(\mathscr{A}^{\#})^{(\alpha_q)}]^i\neq 0\}\\
&=\sup\{i\in \Bbb{Z}|[(\coprod\limits_{j\le 0}\coprod\limits_{b\le q\le u}\Sigma^{-t+j-q}(k^{(\beta_j)})^{(\alpha_q)}]^i\neq 0\}\\
&= t+u= \mathrm{depth}_{\mathscr{A}}\mathscr{A}+ \mathrm{Ext.reg}\, M.
\end{align*}

\end{proof}
\begin{rem}
Note that a homologically smooth DG algebra is not necessarily Gorenstein. For example, the trivial DG free algebra $$\mathscr{A}=(k\langle x,y\rangle, 0)\quad \text{ with}\quad |x|=|y|=1$$ is homologically smooth but not Gorenstein (cf. \cite[Proposition 6.2]{MXYA}). Since there are Noetherian non AS-Gorenstein connected graded algebras with finite global dimension, one sees that homologically smooth DG algebras are not necessarily Gorenstein under the additional assumption that the cohomology graded algebra $H(\mathscr{A})$ is Noetherian.
\end{rem}

\section{some examples}
In this section, we list some homologically smooth and Gorenstein connected cochain DG algebras whose cohomology algebra is Noetherian.
\begin{ex}\label{ex1}
 Let $\mathscr{A}$ be a connected DG algebra such that $\mathscr{A}^{\#} = k\langle x,y\rangle/(xy+yx) $ with $|x|=|y|=1$ and its differential $\partial_{\mathscr{A}}$ is defined by $\partial_{\mathscr{A}}(x) = y^2$ and  $\partial_{\mathscr{A}}(y) = 0$.
 By \cite[Example 3.12]{MW2},  $\mathscr{A}$
is a homologically smooth and  Gorenstein DG algebra with
 $$H(\mathscr{A})\cong
 k[\lceil x\rceil^2, \lceil y\rceil]/(\lceil y\rceil^2).$$
\end{ex}

\begin{ex}\label{ex2}
Let $\mathscr{A}$ be the connected cochain DG algebra such that
$$ \mathscr{A}^{\#}= k \langle x,y\rangle/\left(\begin{array}{ccc}
                            x^2y-(\xi-1) xyx- \xi yx^2            \\
                             xy^2-(\xi-1) yxy-\xi y^2x           \\
                                                 \end {array}\right) $$
is the graded down-up algebra generated by degree $1$ elements $x,y$, and its differential $\partial_{\mathscr{A}}$ is defined by
 $\partial_{\mathscr{A}}(x)=y^2$ and $\partial_{\mathscr{A}}(y)=0$, where $\xi$ is a fixed primitive cubic root of unity.
 By \cite[Proposition 6.1]{MHLX}, $A$ is a Calabi-Yau DG algebra. So $\mathscr{A}$ is a homologically smooth and Gorenstein DG algebra.
 By \cite[Proposition 5.5]{MHLX}, $$H(\mathscr{A})=\frac{k\langle \lceil xy+yx\rceil, \lceil y\rceil\rangle}{\left(\begin{array}{ccc}
                                  \xi \lceil y\rceil \lceil xy+yx\rceil - \lceil xy+yx\rceil \lceil y\rceil \\
                                   \lceil y^2\rceil \\
                                    \end {array}\right)}. $$
\end{ex}

\begin{ex}\label{ex3}
Let $\mathscr{A}$ be the connected cochain DG algebra such that
$$ \mathscr{A}^{\#}= k \langle x,y\rangle/\left(\begin{array}{ccc}
                            x^2y- yx^2            \\
                             xy^2- y^2x           \\
                                                 \end {array}\right) $$
is the graded down-up algebra generated by degree $1$ elements $x,y$, and its differential $\partial_{\mathscr{A}}$ is defined by
 $\partial_{\mathscr{A}}(x)=y^2$ and $\partial_{\mathscr{A}}(y)=0$.
 By \cite{MHLX}, $\mathscr{A}$ is a Calabi-Yau DG algebra with $$H(\mathscr{A})= k[\lceil x^2\rceil,\lceil y\rceil,\lceil xy+yx\rceil]/(\lceil y\rceil^2).$$ Hence $\mathscr{A}$ is a homologically smooth and Gorenstein DG algebra.
 \end{ex}
For the three examples above, the corresponding DG algebras are homologically smooth and Gorenstein DG algebras whose cohomology algebras are Noetherian. We can apply Theorem \ref{mainres} and Theorem \ref{formula} to them.

\subsection*{Acknowledgments}
The author are grateful to Jiwei He, Xingting Wang and James Zhang for helpful discussions.\\
\textbf{Author Contributions} N/A \\
\textbf{Funding}\, The author was supported by NSFC (Grant No.11871326).\\
\textbf{Availability of data and materials}\, All data is available within the article.\\
\\
\begin{Large}\textbf{Declarations}\end{Large}\\
\\
\textbf{Ethical Approval} N/A\\
\textbf{Competing interests}\, The author has no competing interests to declare that are relevant to the content of this
article.


\def\refname{References}

\end{document}